\documentclass[12pt]{amsart}
\usepackage{amsfonts}
\usepackage{amstext}
\usepackage{amsmath}
\usepackage{pstricks, pst-node}
\usepackage{amssymb}
\usepackage{amsthm}
\usepackage[mathscr]{eucal}
\usepackage{graphicx}
\usepackage{amsmath}
\usepackage{hhline}

\psset{unit=1pt, arrowsize=4pt, linewidth=.7pt}
\psset{linecolor=blue}
\newgray{grayish}{.90}
\newrgbcolor{embgreen}{0 .5 0}

\def\vblack(#1, #2)#3{\cnode*[linecolor=black](#1, #2){3}{#3}}
\def\vwhite(#1,#2)#3{\cnode[linecolor=black,fillcolor=white,fillstyle=solid](#1,
#2){3}{#3}}
\countdef\x=23
\countdef\y=24
\countdef\z=25
\countdef\t=26

\def\tbox(#1,#2)#3{
\x=#1 \y=#2
\multiply\x by 12
\multiply\y by 12
\z=\x \t=\y
\advance\z by 12
\advance\t by 12
\psline(\x,\y)(\x,\t)(\z,\t)(\z,\y)(\x,\y)
\advance\x by 6
\advance\y by 6
\rput(\x,\y){{\bf #3}}}

\newcommand{\ket}[1]{\ensuremath{|#1\rangle}}
\newcommand{\bra}[1]{\ensuremath{\langle #1|}}
\newcommand{\T}{\mathcal{T}}
\newcommand{\ttt}{\sigma}

\newcommand{\ct}{\mathcal{T}}
\newcommand{\beq}{\begin{equation}}
\newcommand{\eeq}{\end{equation}}
\newcommand{\sn}{\mathfrak{S}_n}
\newcommand{\bea}{\begin{eqnarray}}
\newcommand{\eea}{\end{eqnarray}}
\DeclareMathOperator{\wt}{wt}
\DeclareMathOperator{\Des}{Des}
\DeclareMathOperator{\type}{type}

\def\llangle{\langle}
\def\rrangle{\rangle}
\def\QStat{\mathrm{QFact}}

\def\i{\mathrm{i}}

\def\dd{\mathsf{d}}
\def\ee{\mathsf{e}}
\def\1{\mathsf{1}}
\def\({\left(}
\def\){\right)}
\def\Z{\mathbb{Z}}

\newtheorem{theorem}{Theorem}[section]

\newtheorem{proposition}[theorem]{Proposition}
\newtheorem{lemma}[theorem]{Lemma}

\newtheorem{example}[theorem]{Example}
\newtheorem{corollary}[theorem]{Corollary}

\newtheorem{remark}[theorem]{Remark}

\newtheorem{problem}[theorem]{Problem}
\newtheorem{definition}[theorem]{Definition}

\begin{document}

\title[Formulae for Askey-Wilson moments and enumeration
of staircase tableaux]{Formulae for Askey-Wilson moments
and enumeration of staircase tableaux}

\author{S. Corteel, R. Stanley, D. Stanton, and L. Williams}
\date{\today}
\thanks{The first author was partially supported by
  the grant ANR-08-JCJC-0011;  the fourth author
  was partially supported by the NSF grant DMS-0854432 and an Alfred Sloan Fellowship.}
\address{LIAFA,
Centre National de la Recherche Scientifique et Universit\'e Paris Diderot,
Paris 7, Case 7014, 75205 Paris Cedex 13
France}
\email{corteel@liafa.jussieu.fr}
\address{Department of Mathematics, Massachusetts Institute of Technology, 
Cambridge, MA 02138}
\email{rstan@math.mit.edu}
\address{Department of Mathematics, University of Minnesota,
Minneapolis, MN 55455}
\email{stanton@math.umn.edu}
\address{Department of Mathematics, University of California, Berkeley,
Evans Hall Room 913, Berkeley, CA 94720}
\email{williams@math.berkeley.edu}


\begin{abstract}
We explain how the moments of the (weight function
of the) Askey Wilson polynomials
are related to the enumeration 
of the {\it staircase tableaux} 
introduced by the first and fourth authors \cite{CW3, CW4}. 
This gives us a direct combinatorial formula
for these moments, which is related to, but more 
elegant than the formula given in \cite{CW3}.
Then we use techniques developed by Ismail and 
the third author to give  explicit formulae
for these moments and for the enumeration of staircase tableaux. 
Finally we study the enumeration of staircase tableaux at various
specializations of the parameterizations; for example, we obtain
the Catalan numbers, Fibonacci numbers, Eulerian numbers, 
the number of permutations, and the number of matchings.
\noindent

[Keywords: staircase tableaux,
asymmetric exclusion process, Askey-Wilson polynomials, permutations,
matchings]\\
\noindent
\end{abstract}

\maketitle
\setcounter{tocdepth}{1}
\tableofcontents


\section{Introduction}\label{intro}

In recent work \cite{CW3, CW4}
the first and fourth 
authors presented a new combinatorial
object that they called {\em staircase tableaux}.
They used these objects to solve two related problems:
to give a combinatorial formula for the stationary distribution
of the asymmetric exclusion process on a one-dimensional
lattice with open boundaries, where all parameters
$\alpha, \beta, \gamma, \delta, q$ are general;
and to give a combinatorial formula for the moments of 
the (weight function of the) Askey-Wilson polynomials.
In this paper we build upon that work, 
and give a somewhat simpler combinatorial
formula for the Askey-Wilson moments.  We also use
work of Ismail and the third author to give an explicit
formula for the Askey-Wilson moments.  Finally we study some special
cases and explore the combinatorial properties of staircase 
tableaux: for example, we highlight a forest structure underlying
staircase tableaux.

\begin{definition}
A \emph{staircase tableau} of size $n$ is a Young diagram of ``staircase"
shape $(n, n-1, \dots, 2, 1)$ such that boxes are either empty or
labeled with $\alpha, \beta, \gamma$, or $\delta$, subject to the following conditions:
\begin{itemize}
\item no box along the diagonal is empty;
\item all boxes in the same row and to the left of a $\beta$ or a $\delta$ are empty;
\item all boxes in the same column and above an $\alpha$ or a $\gamma$ are empty.
\end{itemize}
The \emph{type} $\type(\T)$ of a staircase tableau $\T$
is a word in $\{\circ, \bullet\}^n$
obtained by reading the diagonal boxes from
northeast to southwest and writing a $\bullet$ for each $\alpha$ or $\delta$,
and a $\circ$ for each $\beta$ or $\gamma$.
\end{definition}

See the left of Figure \ref{staircase} for an example.

\begin{figure}[h]
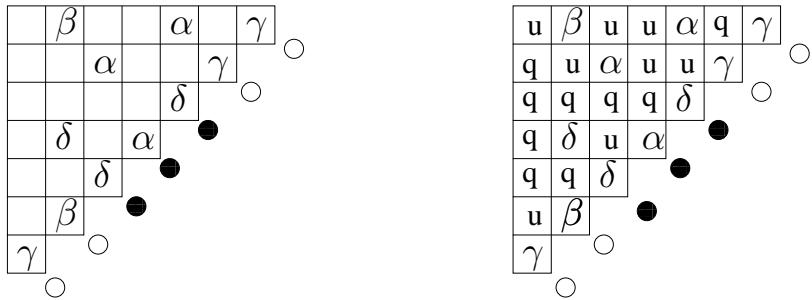

\input{Staircase.pstex_t} \hspace{6em} \input{Staircase2.pstex_t}
\caption{A staircase tableau
of size $7$ and type $\circ \circ \bullet \bullet \bullet \circ \circ$}
\label{staircase}
\end{figure}

Staircase tableaux with no $\gamma$'s or $\delta$'s are
in bijection with permutation tableaux \cite{Postnikov,SW} and 
alternative tableaux \cite{Viennot}.  See \cite{CW4} for more details.

\begin{definition}\label{weight}
The \emph{weight} $\wt(\T)$ of a staircase tableau $\T$ is a monomial in
$\alpha, \beta, \gamma, \delta, q$, and $u$, which we obtain as follows.
Every blank box of $\T$ is assigned a $q$ or $u$, based on the label of the closest
labeled box to its right in the same row and the label of the closest labeled box
below it in the same column, such that:
\begin{itemize}
\item every blank box which sees a $\beta$ to its right gets a $u$;
\item every blank box which sees a $\delta$ to its right gets  a $q$;
\item every blank box which sees an $\alpha$ or $\gamma$ to its right,
  and an $\alpha$ or $\delta$ below it, gets  a $u$;
\item every blank box which sees an $\alpha$ or $\gamma$ to its right,
  and a $\beta$ or $\gamma$ below it, gets a $q$.
\end{itemize}
After filling all blank boxes,
we define $\wt(\T)$
to be the product of all labels in all boxes.
\end{definition}

The right of Figure \ref{staircase} 
shows that the weight of the staircase tableau is
$\alpha^3 \beta^2 \gamma^3 \delta^3 q^9 u^8$.

\begin{remark}
The weight of a staircase tableau
always has degree $n(n+1)/2$.  
For convenience, we will usually set $u=1$,
since this results in no loss of information.
\end{remark}

We define 
$$Z_n(\alpha,\beta,\gamma,\delta;q,u)=\sum_{\T\ {\rm of}\ {\rm size}\ n} {\rm wt(\T)}.
$$
This is the generating polynomial for staircase tableaux of size $n$.
We also use the symbol 
$Z_n(\alpha,\beta,\gamma,\delta;q)$ to denote the same quantity
with $u=1$.

We now review the definition of 
the (partially) asymmetric exclusion process \cite{Derrida1}, a classical
model in 
statistical mechanics.  This is a model of 
particles hopping on a lattice with $n$ sites,
where particles may hop to adjacent sites in the lattice,
and may enter and exit the lattice at both the left and right boundaries,
subject to the condition that at most one particle may occupy a given
site.
The model can be described by
a discrete-time Markov chain \cite{Derrida1,jumping} as follows.

\begin{definition}
Let $\alpha$, $\beta$, $\gamma$, $\delta$,  $q$, and $u$ be constants such that
$0 \leq \alpha \leq 1$, $0 \leq \beta \leq 1$,
$0 \leq \gamma \leq 1$, $0 \leq \delta \leq 1$,
$0 \leq q \leq 1$,
and $0 \leq u \leq 1$.
The ASEP is the Markov chain on the
$2^n$ words in the language
language $\{\circ, \bullet\}^*$,
with
transition probabilities:
\begin{itemize}
\item  If $X = A\bullet \circ B$ and
$Y = A \circ \bullet B$ then
$P_{X,Y} = \frac{u}{n+1}$ (particle hops right) and
$P_{Y,X} = \frac{q}{n+1}$ (particle hops left).
\item  If $X = \circ B$ and $Y = \bullet B$
then $P_{X,Y} = \frac{\alpha}{n+1}$ 
\item  If $X = B \bullet$ and $Y = B \circ$
then $P_{X,Y} = \frac{\beta}{n+1}$ 
\item  If $X = \bullet B$ and $Y = \circ B$
then $P_{X,Y} = \frac{\gamma}{n+1}$ 
\item  If $X = B \circ$ and $Y = B \bullet$
then $P_{X,Y} = \frac{\delta}{n+1}$ 

\item  Otherwise $P_{X,Y} = 0$ for $Y \neq X$
and $P_{X,X} = 1 - \sum_{X \neq Y} P_{X,Y}$.
\end{itemize}
\end{definition}



In the long time limit, the system reaches a steady state where all
the probabilities $P_n(\ttt_1, \ttt_2, \dots , \ttt_n)$ of finding
the system in configuration $\sigma=(\ttt_1, \ttt_2, \dots , \ttt_n)$ are
stationary. Let 
\begin{equation}\label{type}
Z_\sigma(\alpha,\beta,\gamma,\delta;q,u)= \sum_{\T\ {\rm of}\ {\rm type}\  \sigma}\wt(\T).
\end{equation}
This is just the generating polynomial for staircase tableaux of a given type.
As before, we lose no information by setting $u=1$, and in that case
let 
$Z_\sigma(\alpha,\beta,\gamma,\delta;q):= 
Z_\sigma(\alpha,\beta,\gamma,\delta;q,1)$.

\begin{theorem}\cite[Corteel, Williams]{CW3} \label{NewThm}
Consider any state $\sigma$ of
the ASEP with $n$ sites, where the parameters
$\alpha, \beta, \gamma, \delta, q$ and $u$ are general.
Then the steady state probability that the
ASEP is at state $\sigma$ is
precisely
\begin{equation*}
\frac{Z_\sigma(\alpha,\beta,\gamma,\delta;q,u)}{Z_n(\alpha,\beta,\gamma,\delta;q,u)}.
\end{equation*}
\end{theorem}

By Theorem \ref{NewThm}, we can call $Z_n(\alpha,\beta,\gamma,\delta;q,u)$ 
the {\it partition function} of the ASEP.


We now review the definition
of the Askey-Wilson polynomials; these are orthogonal polynomials
with five free parameters $a, b, c, d, q$, 
which reside at the top of the hierarchy of the one-variable
$q$-orthogonal polynomials in the Askey scheme \cite{AW,GR}.

\begin{remark}\label{substitution}
When working with Askey-Wilson polynomials, it will be convenient
to use three variables $x, \theta, z$, which are related to 
each other as follows:
\begin{equation*}
x = \cos \theta,~~~~~
z =e^{i\theta},~~~~~
x = \frac{z+z^{-1}}{2}
\end{equation*}
\end{remark}

\begin{definition}
The Askey-Wilson polynomial $P_n(x)=P_n(x;a,b,c,d\vert q)$
is explicitly defined to be
$$a^{-n}(ab,ac,ad;q)_n\
\sum_{k=0}^n \frac{(q^{-n},q^{n-1}abcd,ae^{i\theta},ae^{-i\theta};q)_k}
{(ab,ac,ad,q;q)_k}q^k,
$$
where $n$ is a non-negative integer and
\begin{eqnarray*}
(a_1,a_2,\cdots,a_s;q)_n=\prod_{r=1}^s \prod_{k=0}^{n-1} (1-a_rq^k).
\end{eqnarray*}
\end{definition}

For $|a|, |b|, |c|, |d| < 1$, 
the orthogonality is expressed by
\begin{eqnarray*}
\oint_C \frac{dz}{4\pi iz} w\left(\frac{z+z^{-1}}{2}\right)
P_m\left(\frac{z+z^{-1}}{2}\right)P_n\left(\frac{z+z^{-1}}{2}\right)
=\frac{h_n}{h_0} \delta_{mn},
\label{eqn:orthoointAW}
\end{eqnarray*}
where the integral contour $C$ is a closed path which
encloses the poles at $z=aq^k$, $bq^k$, $cq^k$, $dq^k$ $(k\in \Z_+)$
and excludes the poles at $z=(aq^k)^{-1}$, $(bq^k)^{-1}$, $(cq^k)^{-1}$,
$(dq^k)^{-1}$ $(k\in \Z_+)$, and where
\begin{align*}
h_0&=h_0(a,b,c,d,q)=
\frac{(abcd;q)_\infty}{(q,ab,ac,ad,bc,bd,cd;q)_\infty},\\
\frac{h_n}{h_0}&=
\frac{(1-q^{n-1}abcd)(q,ab,ac,ad,bc,bd,cd;q)_n}
{(1-q^{2n-1}abcd)(abcd;q)_n} ,\\
w(\cos\theta)&=
\frac{(e^{2i\theta},e^{-2i\theta};q)_\infty}
{h_0 (ae^{i\theta},ae^{-i\theta},be^{i\theta},be^{-i\theta},
ce^{i\theta},ce^{-i\theta},de^{i\theta},de^{-i\theta};q)_\infty}.\\
\end{align*}
(In the other parameter region, the orthogonality is continued analytically.)

\begin{remark}
We remark that our definition of the weight function above
differs slightly from the definition given in \cite{AW}; 
the weight function in \cite{AW} did not have the 
$h_0$ in the denominator.   Our convention simplifies some
of the formulas to come.
\end{remark}

\begin{definition}
The moments of the (weight function of the) Askey-Wilson polynomials --
which we sometimes refer to as simply the \emph{Askey-Wilson moments} --
are defined by 
\begin{eqnarray*}
\mu_n(a,b,c,d\vert q) = \oint_C \frac{dz}{4\pi iz} w\left(\frac{z+z^{-1}}{2}\right)
\left(\frac{z+z^{-1}}{2}\right)^k.
\label{eqn:moment}
\end{eqnarray*}
\end{definition}

The combinatorial formula given in \cite{CW3, CW4} is the following.
\begin{theorem}\cite[Corteel, Williams]{CW3} \label{moments}
The $n$th Askey-Wilson moment is given by
\begin{equation*}
\mu_n(a,b,c,d\vert q) = \sum_{\ell=0}^n (-1)^{n-\ell}{n \choose \ell} 
\left(\frac{1-q}{2}\right)^{\ell}
 \frac{{Z}_{\ell}(\alpha,\beta,\gamma,\delta;q)}{\prod_{j=0}^{\ell-1} (\alpha \beta - \gamma \delta q^j)}, 
\end{equation*}
where \begin{equation}\label{subs1}
\alpha=\frac{1-q}{1+ac+a+c},~~~~~ 
\beta=\frac{1-q}{1+bd+b+d},~~~~~
\gamma=\frac{-(1-q)ac}{1+ac+a+c},~~~~~
\delta=\frac{-(1-q)bd}{1+bd+b+d}.
\end{equation}
\end{theorem}

Note that this formula is not totally 
satisfactory as it has an alternating sum.

In the first half of this paper we give combinatorial and explicit
formulas for Askey-Wilson polynomials and generating functions
of staircase tableaux, including 
a combinatorial
formula for the moments ``on the nose".

To give this  formula, we define $t(\T)$ to be the number of (black) particles
in $\type(\T)$.
For example  the tableau $\T$ in Figure \ref{staircase}
has $t(\T)=3$.  We define the {\it fugacity partition function}
of the ASEP to be 
$$Z_n(y;\alpha,\beta,\gamma,\delta;q)=
\sum_{\T\ {\rm of } \ {\rm size}\ n}\wt(\T)y^{t(\T)},
$$
because this formula is a $y$-analogue of the partition function.

The exponent of $y$ 
keeps track of the number of black particles
in each state.

\begin{theorem}\label{moments2}
The $n^{th}$ Askey-Wilson moment    is 
equal to
$$
\mu_n(a,b,c,d\vert q)= 
\frac{(1-q)^n}{2^ni^n\prod_{j=0}^{n-1}(\alpha\beta-\gamma\delta q^j)}
Z_n(-1;\alpha,\beta,\gamma,\delta;q),
$$
where $i^2=-1$ and
\begin{equation}\label{subs2}
\alpha=\frac{1-q}{1-ac+ai+ci},~~~~~
\beta=\frac{1-q}{1-bd-bi-di},~~~~~
\gamma=\frac{(1-q)ac}{1-ac+ai+ci},~~~~~
\delta=\frac{(1-q)bd}{1-bd-bi-di}.
\end{equation}
\end{theorem}

Note that the Askey-Wilson moments are in general rational 
expressions (with a simple denominator); the coefficients are 
not all positive, but they are all real.  See Example
\ref{moments-example}.  However,
it's not at all clear from  Theorem \ref{moments2}
that the coefficients are real.

Using work of Ismail and the third author \cite{IS},
we also give  explicit 
formulas for both the Askey-Wilson moments and the 
fugacity partition 
function 
of the ASEP.

\begin{theorem}\label{moments3}
The moments $\mu_n(a,b,c,d\vert q)$  are
\begin{eqnarray*}
\frac{1}{2^n}\sum_{k=0}^n \frac{(ab,ac,ad;q)_k}{(abcd;q)_k}q^k
\sum_{j=0}^k q^{-j^2}a^{-2j} \frac{(q^{j}a+q^{-j}/a)^n}{(q,q^{-2j+1}/a^2;q)_{j}(q,q^{2j+1} a^2;q)_{k-j}}.
\end{eqnarray*}
\end{theorem}

\begin{theorem}\label{explicit-Z}
The fugacity partition function $Z_n(y;\alpha,\beta,\gamma,\delta;q)$ of 
the ASEP is
\begin{align*}
\label{Zformula}
Z_n(y;\alpha,\beta,\gamma,\delta;q)
=&(abcd;q)_n \biggl(\frac{\alpha\beta}{1-q}\biggr)^n
\sum_{k=0}^n\frac{(ab,ac/y,ad;q)_k}{(abcd;q)_k}q^k\\
&\times\sum_{j=0}^k q^{-j^2}(a^2/y)^{-j}
\frac{(1+y+q^{j}a+q^{-j}y/a)^n}{(q,q^{-2j+1}y/a^2;q)_{j}
(q,a^2q^{1+2j}/y;q)_{k-j}},
\end{align*}
where 
$$
a=\frac{1-q-\alpha+\gamma+\sqrt{(1-q-\alpha+\gamma)^2+4\alpha\gamma}}{2\alpha},~~~~ 
b=\frac{1-q-\beta+\delta+\sqrt{(1-q-\beta+\delta)^2+4\beta\delta}}{2\beta},
$$
\begin{equation}
c=\frac{1-q-\alpha+\gamma-\sqrt{(1-q-\alpha+\gamma)^2+4\alpha\gamma}}{2\alpha},~~~~ 
d=\frac{1-q-\beta+\delta-\sqrt{(1-q-\beta+\delta)^2+4\beta\delta}}{2\beta}.
\label{eqGPfor}
\end{equation}
\label{GPfor}
(Note that these expressions for $a,b,c,d$  invert the 
transformation given in Theorem \ref{moments}.)
\end{theorem}

In the second half of this paper we explore the wonderful 
combinatorial properties of staircase tableaux.  For example,
when we specialize some of the variables in the generating polynomial
$Z_n(y;\alpha,\beta,\gamma,\delta; q)$ for staircase tableaux, 
we get some nice formulas and combinatorial numbers;
see Table \ref{partition-table} below.
The reference for each statement in the table
is given in the rightmost column.  A few of the simple statements we leave
as exercises.

\begin{table}[h]
\begin{tabular}{|p{.3cm}|p{.3cm}|p{.3cm}|p{.4cm}|p{.4cm}|p{.4cm}|p{6cm}|p{5.8cm}|}
\noalign{\smallskip}
\hline
\noalign{\smallskip}
$\alpha$ & $\beta$ & $\gamma$ & $\delta$ & $q$ & $y$ & $Z_n(y;\alpha,\beta,\gamma,\delta; q)$ & Reference\\
\noalign{\smallskip}
\hline
\noalign{\smallskip}
\hline
\noalign{\smallskip}
$\alpha$ & $\beta$ & $\gamma$ & $\delta$ & $1$ & $1$ & $\prod_{j=0}^{n-1} (\alpha+\beta+\gamma+\delta+j(\alpha+\gamma)(\beta+\delta))$ & Theorem \ref{partition} \\
\noalign{\smallskip}
\hline
\noalign{\smallskip}
$\alpha$ & $\beta$ & $\gamma$ & $-\beta$ & $q$ & $1$ & $ \prod_{j=0}^{n-1} (\alpha+ q^j \gamma)$ & Proposition \ref{d-b}\\
\noalign{\smallskip}
\hline
\noalign{\smallskip}
$\alpha$ & $\beta$ & $\gamma$ & $\beta$ & $q$ & $-1$ & $(-1)^n \prod_{j=0}^{n-1} (\alpha- q^j \gamma)$ & Proposition \ref{y-1}\\
\noalign{\smallskip}
\hline
\noalign{\smallskip}
$\alpha$ & $0$ & $\gamma$ & $0$ & $q$ & $y$ & $\prod_{j=0}^{n-1} (y\alpha+q^j \gamma)$ & Exercise.\\
\noalign{\smallskip}
\hline
\noalign{\smallskip}
$0$ & $\beta$ & $\gamma$ & $0$ & $q$ & $y$ & $\prod_{j=0}^{n-1} (\beta+\beta \gamma [j]_q + \gamma q^j)$ & See \cite{SDH}.  \\
\noalign{\smallskip}
\hline
\noalign{\smallskip}
$\alpha$ & $\beta$ & $0$ & $0$ & $q$ & $y$ & \cite[Theorem 1.3.1]{JV} & See \cite{JV}.\\
\noalign{\smallskip}
\hline
\noalign{\smallskip}
$1$ & $1$ & $0$ & $0$ & $q$ & $y$ & $\sum_{k=1}^{n+1} E_{k,n+1}(q) y^{k-1}$ &
     See \cite[Section 5]{Williams} for 
     definition of $E_{k,n}(q)$; use 
    \cite[Theorem 8]{Corteel}.\\
\noalign{\smallskip}
\hline
\noalign{\smallskip}
$1$ & $1$ & $0$ & $0$ & $-1$ & $y$ & $(y+1)^n$ & 
 \cite[Proposition 5.7]{Williams} \\
\noalign{\smallskip}
\hline
\noalign{\smallskip}
$\alpha$ & $\alpha$ & $\alpha$ & $\alpha$ & $-1$ & $y$ & $0$ \text{ for }$n\geq 3$ &  See Proposition \ref{0}. \\
\noalign{\smallskip}
\hline
\noalign{\smallskip}
$1$ & $1$ & $1$ & $1$ & $1$ & $y$ & $2^n (y+1)^n n!$ & Exercise. \\
\noalign{\smallskip}
\hline
\noalign{\smallskip}
$1$ & $1$ & $1$ & $1$ & $1$ & $1$ & $4^n n! = 4n!!!!$ & Follows from Theorem \ref{partition}. \\
\noalign{\smallskip}
\hline
\noalign{\smallskip}
$1$ & $1$ & $1$ & $0$ & $1$ & $1$ & $(2n+1)!!$ & Follows from Theorem 
\ref{partition}. \\
\noalign{\smallskip}
\hline
\noalign{\smallskip}
$1$ & $1$ & $0$ & $0$ & $1$ & $1$ & $(n+1)!$ & Follows from Theorem \ref{partition}. \\
\noalign{\smallskip}
\hline
\noalign{\smallskip}
$1$ & $1$ & $0$ & $0$ & $0$ & $1$ & $C_{n+1}=\frac{1}{n+2} {2n+2 \choose n+1}$ &
Follows from \cite[Section 5]{Williams}.\\
\noalign{\smallskip}
\hline
\noalign{\smallskip}
$1$ & $1$ & $1$ & $0$ & $0$ & $0$ & $F_{2n+1}$ (Fibonacci) & See Corollary \ref{genfun}. \\
\hline
\noalign{\smallskip}
$1$ & $1$ & $1$ & $0$ & $0$ & $1$ & Sloane A026671 & See Corollary \ref{genfun}. \\
\noalign{\smallskip}
\hline
\noalign{\smallskip}
$1$ & $1$ & $1$ & $1$ & $0$ & $0$ & $2 F_{2n}$ (Fibonacci) &  See Corollary \ref{Fibonacci}. \\
\hline
\end{tabular}
\bigskip
\caption{}
\label{partition-table}
\end{table}

The paper is organized as follows. 
In Section 2, we explain how the Askey-Wilson moments
are related to the generating polynomial
of staircase tableaux. In Section 3, we compute explicit formulas 
for the moments
 and for the generating polynomials of staircase tableaux. 
In Section 4, we study some specializations of the generating polynomials, 
namely
$q=0$, $q=1$ and $\delta=0$. In those cases we highlight the connection
to other combinatorial objects.
We conclude this paper with a list of open problems.

{\bf Acknowledgment.} The authors would like
to thank Philippe Nadeau and Eric Rains
for their constructive suggestions. 

\section{A combinatorial formula for Askey-Wilson moments}
\label{newformula}

The goal of this section is to prove Theorem \ref{moments2}.
Before we do so, we review the connection between 
orthogonal polynomials and tridiagonal matrices.
Recall that by Favard's Theorem, 
orthogonal polynomials satisfy a three-term recurrence.
\begin{theorem}
Let $\{P_k(x)\}_{k \geq 0}$ be a family of monic orthogonal polynomials.  Then
there exist coefficients $\{b_k\}_{k \geq 0}$ and $\{\lambda_k\}_{k \geq 1}$
such that
$P_{k+1}(x) = (x-b_k) P_k(x) - \lambda_k P_{k-1}(x)$.
\end{theorem}

By work of \cite{Fla82,Vie88}, the $n$th moment of a family
of monic orthogonal polynomials can be computed using
a tridiagonal matrix, whose rows contain the information
of the three-term recurrence.
In what follows, $\bra{\widetilde{W}}=(1,0,0,\dots)$ and $\ket{\widetilde{V}}=\bra{\widetilde{W}}^T$.
Note that we use
the bra and ket notations to indicate row and column vectors, respectively. 

\begin{theorem}\cite{Fla82, Vie88}\label{moment-motzkin}
Consider a family of monic orthogonal polynomials $\{P_k(x)\}_{k \geq 0}$
which satisfy the three-term recurrence
$P_{k+1}(x) = (x-b_k) P_k(x) - \lambda_k P_{k-1}(x)$, for
$\{b_k\}_{k \geq 0}$ and $\{\lambda_k\}_{k \geq 1}$.
Then the $n$th moment $\mu_n$ of $\{P_k(x)\}_{k \geq 0}$ is equal to
$\bra{\widetilde{W}} M^n \ket{\widetilde{V}}$, where
$M=(m_{ij})_{i,j\geq 0}$ is the tridiagonal matrix with
rows and columns indexed by the non-negative integers, such that
$m_{i,i-1}=\lambda_i,$ $m_{ii}=b_i$, and $m_{i,i+1}=1$.
\end{theorem}

See \cite{CJW} for a simple proof of Theorem \ref{moment-motzkin}.
We also note the following.

\begin{remark}\label{same-moments}
The polynomials defined by $Q_{k+1}(x)=(x-b_k) Q_k(x)-\lambda_k Q_{k-1}(x)$
have the same moments as the polynomials defined by 
$a_k Q'_{k+1}(x) = (x-b_k) Q'_k(x) - c_k Q'_{k-1}(x)$ as long as 
$a_{k-1} c_k = \lambda_k$.
\end{remark}

Now consider the following tridiagonal matrices,
which were introduced by Uchiyama, Sasamoto and Wadati in 
\cite{USW}.

\begin{eqnarray*}
\dd=\left[
\begin{array}{cccc}
d_0^\natural 	& d_0^\sharp 	& 0	 	& \cdots\\
d_0^\flat 	& d_1^\natural 	& d_1^\sharp 	& {}\\
0 		& d_1^\flat 	& d_2^\natural 	& \ddots\\
\vdots 		& {} 		& \ddots	& \ddots
\end{array}
\right] ,
&&\qquad 
\ee=\left[
\begin{array}{cccc}
e_0^\natural 	& e_0^\sharp 	& 0 		& \cdots\\
e_0^\flat 	& e_1^\natural 	& e_1^\sharp 	& {}\\
0 		& e_1^\flat 	& e_2^\natural 	& \ddots\\
\vdots 		& {}		& \ddots	& \ddots
\end{array}
\right], \text{ where }
\label{eqn:repde2}
\end{eqnarray*}
\begin{eqnarray*}
d_n^\natural:= d_n^\natural(a,b,c,d) &=&
\frac{q^{n-1}}{(1-q^{2n-2}abcd)(1-q^{2n}abcd)}\nonumber\\
&&\times[
bd(a+c)+(b+d)q-abcd(b+d)q^{n-1}-\{ bd(a+c)+abcd(b+d)\} q^n\nonumber\\
&&-bd(a+c)q^{n+1}+ab^2 cd^2(a+c) q^{2n-1}+abcd(b+d)q^{2n} ] ,\\
e_n^\natural:=e_n^\natural(a,b,c,d) &=&
\frac{q^{n-1}}{(1-q^{2n-2}abcd)(1-q^{2n}abcd)}\nonumber\\
&&\times[
ac(b+d)+(a+c)q-abcd(a+c)q^{n-1}-\{ ac(b+d)+abcd(a+c)\} q^n\nonumber\\
&&-ac(b+d)q^{n+1}+a^2 bc^2 d(b+d) q^{2n-1}+abcd(a+c)q^{2n} ] ,
\label{eqn:elements}
\end{eqnarray*}
\begin{eqnarray*}
&&d_n^\sharp := d_n^\sharp(a,b,c,d) =
\frac{1}{1-q^nac}\mathcal{A}_n ,\qquad
e_n^\sharp:=e_n^\sharp(a,b,c,d) =
-\frac{q^nac}{1-q^nac}\mathcal{A}_n ,\\
&&d_n^\flat:=d_n^\flat(a,b,c,d) =
-\frac{q^nbd}{1-q^nbd}\mathcal{A}_n ,\qquad
e_n^\flat:=e_n^\flat(a,b,c,d) =
\frac{1}{1-q^nbd}\mathcal{A}_n , \text{ and }
\end{eqnarray*}
\begin{align*}
\mathcal{A}_n:=&\mathcal{A}_n(a,b,c,d) \\
=&\left[
\frac{(1-q^{n-1}abcd)(1-q^{n+1})
(1-q^nab)(1-q^nac)(1-q^nad)(1-q^nbc)(1-q^nbd)(1-q^ncd)}
{(1-q^{2n-1}abcd)(1-q^{2n}abcd)^2(1-q^{2n+1}abcd)}
\right]^{1/2} .
\end{align*}

\begin{remark}\label{matrix-moments}
These matrices have the property that the coefficients in the 
$n$th row of $\dd + \ee$ are the coefficients in the three-term recurrence
for the Askey-Wilson polynomials \eqref{AWeq}.  
That three-term recurrence is given by 
\begin{equation}
A_nP_{n+1}(x)+B_nP_n(x)+C_nP_{n-1}(x)=2xP_n(x),
\label{AWeq}
\end{equation}
with $P_0(x)=1$ and $P_{-1}(x)=0$,
where
\begin{align*}
A_n&=
\frac{1-q^{n-1}abcd}{(1-q^{2n-1}abcd)(1-q^{2n}abcd)},
\\
B_n&=
\frac{q^{n-1}}{(1-q^{2n-2}abcd)(1-q^{2n}abcd)}
[(1+q^{2n-1}abcd)(qs+abcds')-q^{n-1}(1+q)abcd(s+qs')],
\\
C_n&=
\frac{(1-q^n)(1-q^{n-1}ab)(1-q^{n-1}ac)(1-q^{n-1}ad)(1-q^{n-1}bc)
(1-q^{n-1}bd)(1-q^{n-1}cd)}
{(1-q^{2n-2}abcd)(1-q^{2n-1}abcd)},
\end{align*}

\begin{eqnarray*}
~~s=a+b+c+d, \qquad s'=a^{-1}+b^{-1}+c^{-1}+d^{-1}.
\end{eqnarray*}

It's now a direct
consequence of Theorem \ref{moment-motzkin} and Remark \ref{same-moments}
that the $n$th moment 
of the Askey-Wilson
polynomials is given by 
$\bra{\widetilde{W}}  (\dd+\ee)^n  \ket{\widetilde{V}}$.
\end{remark}

Also define
\begin{align*}
D_n^\natural&=d_n^\natural(\frac{a}{\sqrt{y}}, b \sqrt{y}, \frac{c}{\sqrt{y}}, d \sqrt{y}),~~~~~
&E_n^\natural=e_n^\natural(\frac{a}{\sqrt{y}}, b \sqrt{y}, \frac{c}{\sqrt{y}}, d \sqrt{y}),\\
D_n^\sharp&=d_n^\sharp(\frac{a}{\sqrt{y}}, b \sqrt{y}, \frac{c}{\sqrt{y}}, d \sqrt{y}),~~~~~
&E_n^\sharp=e_n^\sharp(\frac{a}{\sqrt{y}}, b \sqrt{y}, \frac{c}{\sqrt{y}}, d \sqrt{y}),\\
D_n^\flat&=d_n^\flat(\frac{a}{\sqrt{y}}, b \sqrt{y}, \frac{c}{\sqrt{y}}, d \sqrt{y}),~~~~~
&E_n^\flat=e_n^\flat(\frac{a}{\sqrt{y}}, b \sqrt{y}, \frac{c}{\sqrt{y}}, d \sqrt{y}).
\end{align*}

\begin{lemma}\label{lem1}
$y d_n^\natural + e_n^\natural = \sqrt{y} (D_n^{\natural}+E_n^{\natural}).$
\end{lemma}

\begin{proof}
This follows from the fact that 
$D_n^\natural = \sqrt{y} d_n^{\natural}$ and 
$E_n^\natural = \frac{1}{\sqrt{y}} e_n^{\natural}$.
\end{proof}

\begin{lemma}\label{lem2}
$(yd_n^\sharp+e_n^\sharp)(yd_n^\flat + e_n^\flat) =y(D_n^\sharp + E_n^\sharp)(D_n^\flat+E_n^\flat)$.
\end{lemma}

\begin{proof}
First observe that 
\begin{align*}
y(D_n^\sharp + E_n^\sharp) &= 
\frac{y-q^n ac}{1-q^n a c y^{-1}} \mathcal{A}_n(\frac{a}{\sqrt{y}}, b \sqrt{y}, \frac{c}{\sqrt{y}}, d \sqrt{y}),\\
D_n^\flat + E_n^\flat &= \frac{1-q^n b d y}{1-q^n bdy} \mathcal{A}_n(\frac{a}{\sqrt{y}}, b \sqrt{y}, \frac{c}{\sqrt{y}}, d \sqrt{y}), \text{ and}\\
\mathcal{A}_n(\frac{a}{\sqrt{y}}, b \sqrt{y}, \frac{c}{\sqrt{y}}, d \sqrt{y}) &=
\frac{(1-q^n a c y^{-1})(1-q^n b d y)}{(1-q^n a c)(1-q^n b d)}^{1/2} \mathcal{A}_n(a,b,c,d).
\end{align*}

Multiplying the first two equations gives 
$$y(D_n^\sharp + E_n^\sharp) (D_n^\flat + E_n^\flat) = 
\frac{(y-q^n ac)(1-q^n bdy)}{(1-q^n acy^{-1})(1-q^n bdy)} (\mathcal{A}_n(\frac{a}{\sqrt{y}}, b \sqrt{y}, \frac{c}{\sqrt{y}}, d \sqrt{y}))^2.$$
Then using the third equation we have 
$$y(D_n^\sharp + E_n^\sharp) (D_n^\flat + E_n^\flat) = 
\frac{(y-q^n ac)(1-q^n bdy)}{(1-q^n ac)(1-q^n bd)} \mathcal{A}_n^2.$$
It remains to see that 
$$(yd_n^\sharp+e_n^\sharp)(yd_n^\flat + e_n^\flat)  = \frac{(y-q^n ac)(1-q^n bdy)}{(1-q^n ac)(1-q^n bd)} \mathcal{A}_n^2,$$ but this follows from the definition of 
$d_n^\sharp, e_n^\sharp, d_n^\flat, e_n^\flat$.

\end{proof}

We now define matrices 
$\widetilde{D}$ and $\widetilde{E}$ by 
\begin{equation}
\widetilde{D} =\frac{1}{1-q}(\1+\dd),\qquad \widetilde{E} =\frac{1}{1-q}(\1+\ee).
\label{eqn:defde}
\end{equation}

Then 
$y\widetilde{D} + \widetilde{E} = \frac{1}{1-q} (\1 + y\1 + y \dd + \ee).$

\begin{proposition} \label{thm1} The $n$th moment of the specialization of the Askey-Wilson polynomials
$P_m(x+\frac{\frac{1}{\sqrt{y}} + \sqrt{y}}{2}; \frac{a}{\sqrt{y}}, b \sqrt{y}, \frac{c}{\sqrt{y}}, d\sqrt{y} \vert q)$ is equal to 
$$\frac{\llangle \widetilde{W} \vert (y \widetilde{D}+\widetilde{E})^n \vert \widetilde{V} \rrangle (1-q)^n}{2^n \sqrt{y}^n}.$$
\end{proposition}

\begin{proof}
Note that $$\llangle \widetilde{W} \vert (y \widetilde{D}+\widetilde{E})^n \vert \widetilde{V} \rrangle
= \frac{1}{(1-q)^n} \llangle \widetilde{W} \vert (\1 + y \1 + y\dd + \ee)^n \vert \widetilde{V} \rrangle,$$ and so 
\begin{equation}
\label{sqrt-mess}
\llangle \widetilde{W} \vert (y \widetilde{D}+\widetilde{E})^n \vert \widetilde{V} \rrangle (1-q)^n \sqrt{y}^{-n}
= \llangle \widetilde{W} \vert (\frac{1}{\sqrt{y}}\1 + \sqrt{y} \1 + \sqrt{y}\dd + \frac{1}{\sqrt{y}}\ee)^n \vert \widetilde{V} \rrangle.
\end{equation}

It's easy to see that the right-hand-side of equation (\ref{sqrt-mess}) is the 
$n$th moment for the monic polynomials
$q_m(x+\frac{1}{\sqrt{y}} + \sqrt{y})$, where the $q_m$'s are defined by 
the three-term recurrence 
$x q_n(x) =  q_{n+1}(x) + B'_n q_n(x) + A'_{n-1} C'_n q_{n-1}(x)$, 
and where $A'_n, B'_n, C'_n$ are given by the 
$n$th row of the tridiagonal
matrix 
$\sqrt{y} \dd + \frac{1}{\sqrt{y}} \ee$.

Alternatively, letting $Q_n(x)=q_n(2x)$, 
we can interpret the right-hand-side of (\ref{sqrt-mess}) as 
$2^n$ times the $n$th moment for the non-monic polynomials
$Q_m(x + \frac{\frac{1}{\sqrt{y}} + \sqrt{y}}{2})$, which are defined by the recurrence
$$2x Q_n(x) = Q_{n+1}(x) + B'_n Q_n(x) + A'_{n-1} C'_n Q_{n-1}(x).$$ 

By Lemmas \ref{lem1} and \ref{lem2}, 
\begin{align*}
B'_n &= \frac{y d_n^\natural + e_n^\natural}{\sqrt{y}} = D_n^\natural + E_n^\natural,\text{ and }\\
A'_{n-1} C'_n &= \frac{y d_n^\sharp + e_n^\sharp}{\sqrt{y}} \frac{yd_n^\flat + e_n^\flat}{\sqrt{y}} = (D_n^\sharp + E_n^\sharp)(D_n^\flat + E_n^\flat).
\end{align*}

Note also that by Remark \ref{same-moments}, the polynomials defined by 
$2x Q_n(x) = Q_{n+1}(x) + B'_n Q_n(x) + A'_{n-1} C'_n Q_{n-1}(x)$  and the polynomials 
defined by 
$2x Q'_n(x) = A'_n Q'_{n+1}(x) + B'_n Q'_n(x) + C'_n Q'_{n-1}(x)$ have the same moments. 

Therefore by Remark \ref{matrix-moments},
the $n$th moment of the polynomials $Q_m(x)$ is the $n$th moment of the 
Askey-Wilson polynomials $P_m(x)$, with the specialization 
$a \to \frac{a}{\sqrt{y}},$ 
$b \to b \sqrt{y},$
$c \to \frac{c}{\sqrt{y}},$ and 
$d \to d \sqrt{y}$.  The proposition follows.
\end{proof}

Now we would like to relate this to the matrices $D, E, \vert V\rrangle, \llangle {W}\vert$ given in 
\cite[Definition 6.1]{CW4}, 
which have a combinatorial interpretation in terms of 
staircase tableaux. We do not need their definitions here, but only the following property.
\begin{theorem}\cite[Corteel, Williams]{CW3, CW4}\label{th:refine}
We have that $$
Z_n(y;\alpha,\beta,\gamma,\delta;q)=\llangle {W}\vert (yD+E)^n\vert V\rrangle.
$$
Additionally, the coefficient of $y^i$ above is proportional
to the probability that in the asymmetric exclusion process on $n$ sites,
exactly $i$ sites are occupied by a particle.
\end{theorem}

The following result gives an explicit relation between the moments
of the Askey-Wilson polynomials and the fugacity partition function
of the ASEP.

\begin{corollary}\label{spec:cor}
The $n$th moment of the specialization of the Askey-Wilson polynomials
$P_m(2\sqrt{y} x+1+y; \frac{a}{\sqrt{y}}, b \sqrt{y}, \frac{c}{\sqrt{y}}, d\sqrt{y} \vert q)$ is equal to 
$$\frac{(1-q)^n} 
{\prod_{j=0}^{n-1}(\alpha \beta - \gamma \delta q^j)}  \llangle {W} \vert (y {D}+{E})^n \vert {V} \rrangle
=\frac{(1-q)^n} 
{\prod_{j=0}^{n-1}(\alpha \beta - \gamma \delta q^j)}
Z_n(y;\alpha,\beta,\gamma,\delta;q),$$
 where $\alpha, \beta, \gamma, \delta$ are given by 
(\ref{subs1}).
\label{pfunc1}
\end{corollary}

\begin{proof}
The matrices $\widetilde{D}, \widetilde{E}, \widetilde{V}, \widetilde{W}$ satisfy 
the Matrix Ansatz of Derrida-Evans-Hakim-Pasquier, 
see \cite[Theorem 5.1]{CW4}.  
On the other hand, the matrices $D, E, V, W$ satisfy the modified Matrix Ansatz
of \cite[Theorem 5.2]{CW4}.  By
\cite[Lemma 7.1]{CW4} and the proof of \cite[Theorem 4.1]{CW4}, these can be related via
$$\llangle \widetilde{W} \vert (y\widetilde{D}+\widetilde{E})^n \vert \widetilde{V} \rrangle 
=  \frac{\llangle W \vert (yD+E)^n \vert V \rrangle}{\prod_{j=0}^{n-1} (\alpha \beta - \gamma \delta q^j)}.$$
Now the proof follows from  Proposition
\ref{thm1},\footnote{Note
that Uchiyama-Sasamoto-Wadati \cite{USW} defined vectors 
$\widetilde{W}$ and $\widetilde{V}$ to be 
$h_0^{1/2}(1,0,0,\cdots )$ and 
$h_0^{1/2}(1,0,0,\cdots )^T$, not  
$(1,0,0,\cdots )$ and 
$(1,0,0,\cdots )^T$ as we have done here.  However, their weight function $w$
did not have the factor of $h_0$ as ours has, and these two discrepancies
``cancel each other out."} together with the observation
that $2^n \sqrt{y}^n$ times the $n$th moment of 
the polynomials 
$P_m(x+\frac{\frac{1}{\sqrt{y}} + \sqrt{y}}{2}; \frac{a}{\sqrt{y}}, b \sqrt{y}, \frac{c}{\sqrt{y}}, d\sqrt{y} \vert q)$ is equal to 
the $n$th moment of the polynomials 
$P_m(2\sqrt{y} x+1+y; \frac{a}{\sqrt{y}}, b \sqrt{y}, \frac{c}{\sqrt{y}}, d\sqrt{y} \vert q)$. 
\end{proof}

Corollary \ref{pfunc1} is equivalent to the following one:
\begin{corollary}
The fugacity partition function $Z_n(y;\alpha,\beta,\gamma,\delta;q)$ is equal to $$
(abcd)_n\sqrt{y}^n(\alpha\beta)^n\times \mu_n
$$
where $\mu_n$ are the moments of the orthogonal polynomials
defined by $$
(x-b_n)G_n(x)=G_{n+1}(x)+\lambda_n G_{n-1}(x)
$$
where $$b_n=\frac{1/\sqrt{y}+\sqrt{y}+B_n}{1-q}\ \ \ {\rm and}\ \ \ \lambda_n=\frac{A_{n-1}C_n}{(1-q)^2}$$
and $A_n,B_n,C_n$  are the coefficients of the 3-term
recurrence of the Askey Wilson given in \eqref{AWeq} 
with $a\rightarrow a/\sqrt{y},\ b\rightarrow b\sqrt{y},\
c\rightarrow\ c/\sqrt{y},\ d\rightarrow\ d\sqrt{y}$
and $a,b,c,d$ are given by \eqref{eqGPfor}.
\label{Dennis}
\end{corollary}

If we set $y=-1$ in Corollary \ref{spec:cor}
%
and perform a simple change of variables taking 
$-ai \to a, bi \to b, -ci \to c, di \to d$, we give a formula for the 
Askey-Wilson moments ``on the nose."
\begin{corollary}\label{cor:imag}
The $n$th moment of the Askey-Wilson polynomials
$P_m(x; a, b, c, d \vert q)$ is equal to
$$
\mu_n(a,b,c,d\vert q)=\frac{(1-q)^n}{2^n i^n 
\prod_{j=0}^{n-1}(\alpha \beta - \gamma \delta q^j)}  \llangle {W} \vert (- {D}+{E})^n \vert {V} \rrangle
,$$ where $\alpha, \beta, \gamma, \delta$ are given by (\ref{subs2}).
\end{corollary}

This finishes the proof of Theorem \ref{moments2}, because 
Theorem \ref{moments2} is equivalent to Corollary \ref{cor:imag}
by setting $y=-1$ in Theorem \ref{th:refine}.

\section{Explicit formulae for Askey-Wilson moments and staircase tableaux}





In this section we will give some explicit formulas for the moments
of the Askey-Wilson polynomials. We first  prove a more general statement.
Recall Remark \ref{substitution}.
\begin{proposition}\label{expansion}
Let $p(x)$ be a degree $n$ polynomial in $x$. Then 
\begin{equation}
\label{polymom}
\begin{gathered}
\oint_C{\frac{p(x)w(x,a,b,c,d\vert q)d{z}}{4\pi iz}}
=\sum_{k=0}^n\frac{(ab,ac,ad;q)_k}{(abcd;q)_k}q^k
\sum_{j=0}^k 
\frac{q^{-j^2}a^{-2j}
 p(\frac{q^{j}a+q^{-j}/a}{2})}{(q,q^{-2j+1}/a^2;q)_{j}(q,a^2q^{1+2j};q)_{k-j}}.
\end{gathered}
\nonumber
\end{equation}
\end{proposition}

Recall the Askey-Wilson weight function $w$ defined in Section \ref{intro}.
Let $\phi_n(x; a) = (ae^{i\theta},ae^{-i\theta};q)_n$; this is
a polynomial in $x$ of degree $n$.
Note that 
\begin{equation*}
w(x,a,b,c,d\vert q)\phi_n(x; a)=w(x,aq^n,b,c,d\vert q) \frac{h_0(aq^n,b,c,d,q)}{h_0(a,b,c,d,q)}.
\end{equation*}

Therefore
\begin{lemma}\label{lem-weight}
$$
\oint_C{\frac{\phi_n(x; a) w(x,a,b,c,d\vert q)d{z}}{4\pi i z}}=
\frac{h_0(aq^n,b,c,d,q)}{h_0(a,b,c,d,q)}.
$$
\end{lemma}

Our strategy for proving Proposition \ref{expansion} will be to expand
$f(x)$ in the basis $\phi_n(x; a)$  
by using 
a result of 
Ismail and the third author \cite{IS}, 
and then to apply Lemma \ref{lem-weight}.

\begin{theorem}\cite[Theorem 1.1]{IS}
If we write the degree $n$ polynomial 
$p(x)$ as $\sum_{k=0}^n p_k \phi_k(x;a)$ then
$$
p_k=\frac{(q-1)^k}{(2a)^k(q;q)_k}q^{-\frac{k(k-1)}{4}}(D_q^k p)(x_k),
$$
where 
$$
(D_q^k p)(x)=\frac{2^kq^{\frac{k(1-k)}{4}}}{(q^{1/2}-q^{-1/2})^k}\sum_{j=0}^k \left[\begin{array}{c}
k\\j\end{array}\right]_q
\frac{q^{j(k-j)}z^{2j-k}\check{p}(q^{(k-2j)/2}z)}{(q^{1+k-2j}z^2;q)_j(q^{1-k+2j}z^{-2};q)_{k-j}}, 
$$
$x_k=(aq^{k/2}+a^{-1}q^{-k/2})/2$, $x=\cos \theta$, $z=e^{i\theta}$,
$\left[\begin{array}{c}
k\\j\end{array}\right]_q=\frac{(q;q)_k}{(q;q)_j(q;q)_{k-j}}$, and 
$\check{p}(x)=f(\frac{x+x^{-1}}{2})$.
\end{theorem}

We can now prove Proposition \ref{expansion}.
\begin{proof}
Write $p(x)=\sum_{k=0}^n p_k \phi_k (x;a)$. Then
when $x_k=(aq^{k/2}+a^{-1}q^{-k/2})/2$, we get $z_k=aq^{k/2}$ and
\begin{eqnarray*}
p_k&=& \frac{(q-1)^k}{(2a)^k(q;q)_k}q^{-\frac{k(k-1)}{4}}\frac{2^kq^{\frac{k(1-k)}{4}}}{(q^{1/2}-q^{-1/2})^k}\sum_{j=0}^k \left[\begin{array}{c}
k\\j\end{array}\right]_q
\frac{q^{j(k-j)}(aq^{k/2})^{2j-k}\check{p}(aq^{k-j})}{(q^{1+2k-2j}a^2;q)_j(q^{-1-2k+2j}a^{-2};q)_{k-j}}\\
&=&\frac{q^k q^{k(1-k)/2}}{a^k}\sum_{j=0}^k
\frac{1}{(q;q)_j(q;q)_{k-j}}\frac{q^{j(k-j)}(aq^{k/2})^{2j-k}\check{p}(aq^{k-j})}{(q^{1+2k-2j}a^2;q)_j(q^{-1-2k+2j}a^{-2};q)_{k-j}}\\
&=&q^k\sum_{j=0}^k
\frac{q^{-(k-j)^2}a^{2j-2k}\check{p}(aq^{k-j})}{(q,q^{1+2k-2j}a^2;q)_j(q,q^{-1-2k+2j}a^{-2};q)_{k-j}}.\\
\end{eqnarray*}
Note that 
$$\frac{h_0(aq^k,b,c,d,q)}{h_0(a,b,c,d,q)}=\frac{(ab,ac,ad;q)_k}{(abcd;q)_k}.$$
Therefore
\begin{eqnarray*}
\oint_C{\frac{p(x)w(x,a,b,c,d\vert q) d{z}}{4\pi i z}}&=&\sum_{k=0}^n p_k \frac{h_0(aq^k,b,c,d,q)}{h_0(a,b,c,d,q)}\\
&=& \sum_{k=0}^n \frac{(ab,ac,ad;q)_k}{(abcd;q)_k}q^k\sum_{j=0}^k
\frac{q^{-(k-j)^2}a^{2j-2k}p((aq^{k-j}+a^{-1}q^{j-k})/2)}{(q,q^{1+2k-2j}a^2;q)_j(q,q^{1-2k+2j}a^{-2};q)_{k-j}}\\
&=& \sum_{k=0}^n \frac{(ab,ac,ad;q)_k}{(abcd;q)_k}q^k\sum_{j=0}^k
\frac{q^{-j^2}a^{-j}p((aq^{j}+a^{-1}q^{-j})/2)}{(q,q^{1+2j}a^2;q)_{k-j}(q,q^{1-2j}a^{-2};q)_{j}}.
\end{eqnarray*}
\end{proof}

We can now prove Theorem \ref{moments3}.
\begin{proof}
Setting $p(x)=x^n$ in Proposition \ref{expansion}, we obtain
$$
\mu_n(a,b,c,d\vert q)
=\frac{1}{2^n}\sum_{k=0}^n\frac{(ab,ac,ad;q)_k}{(abcd;q)_k}q^k\\
\times\sum_{j=0}^k q^{-j^2}a^{-2j}
\frac{(q^{j}a+q^{-j}/a)^n}{(q,q^{-2j+1}/a^2;q)_{j}(q,a^2q^{1+2j};q)_{k-j}}.
$$
\end{proof}


\begin{example}\label{moments-example}
\begin{align*}
\mu_1(a,b,c,d)=&
(-a - b - c -d + abc  + abd + acd + bcd)/
 (2(-1 + abcd))\\
\mu_2(a,b,c,d)=&(1 + a^2 + ab + b^2 + ac + bc - a^2bc - ab^2c +
  c^2 - abc^2 + ad \\&+ bd - a^2bd - ab^2d + cd - 
  a^2cd - 4abcd - b^2cd + a^2 b^2cd\\& - ac^2d
    - bc^2d + a^2bc^2d + ab^2c^2d + d^2 - abd^2 - 
  acd^2 - bcd^2\\& + a^2bcd^2 + ab^2cd^2
  + abc^2d^2 - a^2b^2c^2d^2 - q + abq + acq + 
  bcq\\& - a^2bcq - ab^2cq - abc^2q + 
  a^2b^2c^2q + adq + bdq - a^2bdq\\& - ab^2dq + 
  cdq - a^2cdq - 4abcdq - b^2cdq + 
 a^2b^2cdq - ac^2dq\\& - bc^2dq + a^2bc^2dq + 
  ab^2c^2dq - abd^2q + a^2b^2d^2q\\& - acd^2q - 
  bcd^2q + a^2bcd^2q + ab^2cd^2q + 
  a^2c^2d^2q + abc^2d^2q\\&  + b^2c^2d^2q + 
  a^2b^2c^2d^2q)/(4(-1 + abcd)(-1 + abcdq)).
\end{align*}
\end{example}

We also use Proposition \ref{expansion} to prove 
Theorem \ref{explicit-Z}.

\begin{proof}
Now we use the result of Corollary \ref{spec:cor}.
To get the fugacity
 partition function of the ASEP or equivalently the generating
polynomial of staircase tableaux, we have to take  $p(x)=(1+y+2\sqrt{y}x)^n$ 
and substitute
$$
a\rightarrow a/\sqrt{y}, \qquad b\rightarrow b\sqrt{y}, \qquad c\rightarrow c/\sqrt{y}, \qquad d\rightarrow d\sqrt{y}
$$ 
in Proposition \ref{expansion}.
\end{proof}

\begin{example}
\begin{align*}
Z_1=& \alpha y +\delta y+ \beta  + \gamma \\
Z_2=&\alpha^2y^2 + \alpha \delta y^2+ \alpha^2\delta y^2+ \alpha \beta \delta y^2+  \alpha \delta^2 y^2+  \alpha \delta \gamma y^2+  \alpha \delta q y^2+ 
   \delta^2 q y^2+ \alpha \beta y + \alpha^2 \beta y + \alpha \beta^2 y + \\&
\beta \delta y + 
  \alpha \beta \delta y + \alpha \gamma y + \alpha \beta \gamma y + \delta \gamma y + \alpha \delta \gamma y + \beta \delta \gamma y + 
  \delta^2 \gamma y + \delta \gamma^2 y + \alpha \beta q y + \beta \delta q y + \\& \alpha \gamma q y +
  \delta \gamma q y + \beta^2 + \beta \gamma + \alpha \beta \gamma  + \beta^2 \gamma  + 
  \beta \delta \gamma  + \beta \gamma^2  +\beta \gamma q  + \gamma^2 q. 
\end{align*}
\end{example}
\vskip10pt

\subsection{Askey Wilson moments and the partition function  when $q=0$}

If $q=0$ the moments may be computed in another way, using a contour
integral and the residue calculus.  Recall 
the substitutions from Remark \ref{substitution}.

\begin{proposition}\label{prop-residue}
Let $p(x)$ be any polynomial in $x$, and let
$f(z,a,b,c,d)=(1-az)(1-a/z)(1-bz)(1-b/z)(1-cz)(1-c/z)(1-dz)(1-d/z).$
 Then
\begin{equation}
\label{q=0mom}
\begin{aligned}
\oint_C&{\frac{p(x)w(x,a,b,c,d\vert 0) d{z}}{4\pi i z }}=
\frac{-1}{2}\frac{(1-ab)(1-ac)(1-ad)(1-bc)(1-bd)(1-cd)}{1-abcd}\\
\times &\biggl( \frac{p(\frac{a+1/a}{2})(a-1/a)^2}
{(1-a^2)(1-ab)(1-b/a)(1-ca)(1-c/a)(1-da)(1-d/a)}\\
&+\frac{p(\frac{b+1/b}{2})(b-1/b)^2}{(1-b^2)(1-ab)(1-a/b)(1-cb)(1-c/b)(1-db)(1-d/b)}\\
&+\frac{p(\frac{c+1/c}{2})(c-1/c)^2}{(1-c^2)(1-ac)(1-a/c)(1-cb)(1-b/c)(1-dc)(1-d/c)}\\
&+\frac{p(\frac{d+1/d}{2})(d-1/d)^2}{(1-d^2)(1-ad)(1-a/d)(1-db)(1-b/d)(1-dc)(1-c/d)}\\
&+Res \biggl(\frac{p(\frac{z+1/z}{2})(z-1/z)^2}{zf(z,a,b,c,d)},z=0 \biggr)
\biggr)
\end{aligned}
\nonumber
\end{equation}
\end{proposition}

\begin{proof}
Assume that $|a|, |b|, |c|, |d|<1$; these 
conditions are not necessary later. Using the Cauchy Residue
Theorem, we get 
\begin{equation}\label{residue}
\oint_C \frac{p(x) w(x,a,b,c,d \vert 0) d{z}}{4\pi i z}
= \frac{1}{2} \sum_k Res \left(\frac{p(x)w(x,a,b,c,d\vert 0)}{z},
  z=a_k \right),
\end{equation}
where the $a_k$ are the poles inside $C$.

Note that at $q=0$, we have
\begin{align*}
h_0(a,b,c,d,0) &= \frac{1-abcd}{(1-ab)(1-ac)(1-ad)(1-bc)(1-bd)(1-cd)}, \text{ and }\\
w(\cos \theta,a,b,c,d\vert 0)&=\frac{-(z-1/z)^2}{h_0(a,b,c,d,0) f(z,a,b,c,d)}.
\end{align*}
There are five poles inside $C$: $z=a,b,c,d$ and $0.$
 Substituting into (\ref{residue}) gives the result.
\end{proof}

Let $H_n(a,b,c,d)$ be the homogeneous symmetric function of degree $n$ 
in the $8$ variables
$a$, $b$, $c$, $d$, $1/a$, $1/b$, $1/c$, $1/d$.

\begin{theorem} 
\label{q0partfunc}
The partition function $Z_n(y;\alpha,\beta,\gamma,\delta;0)$ is
$$
\begin{aligned}
&\frac{-(\alpha\beta)^n}{2}(1-AB)(1-AC)(1-AD)(1-BC)(1-BD)(1-CD)\\
&\times \biggl(\frac{(1+(1/A+A)\sqrt{y}+y)^n(A-1/A)^2}
{(1-A^2)(1-AB)(1-B/A)(1-CA)(1-C/A)(1-DA)(1-D/A)}\\
&+\frac{(1+(1/B+B)\sqrt{y}+y)^n(B-1/B)^2}
{(1-B^2)(1-AB)(1-A/B)(1-CB)(1-C/B)(1-DB)(1-D/B)}\\
&+\frac{(1+(1/C+C)\sqrt{y}+y)^n(C-1/C)^2}
{(1-C^2)(1-AC)(1-A/C)(1-CB)(1-B/C)(1-DC)(1-D/C)}\\
&+\frac{(1+(1/D+D)\sqrt{y}+y)^n(D-1/D)^2}
{(1-D^2)(1-AD)(1-A/D)(1-DB)(1-B/D)(1-DC)(1-C/D)}\\
&+\frac{1}{ABCD}\sum_{j=0}^n\sum_{k=0}^n\binom{n}{k}\binom{n}{j}
\sqrt{y}^{n+k-j}\\
&\times(H_{n-2-k-j}(A,B,C,D)-2H_{n-4-k-j}(A,B,C,D)+H_{n-6-k-j}(A,B,C,D)
\biggr)
\end{aligned}
$$
where $A=a/\sqrt{y}$, $B=b\sqrt{y}$, $C=c/\sqrt{y}$, $D=d\sqrt{y}$ and
$a,b,c,d$ as in Proposition \ref{GPfor}.
\end{theorem}
\begin{proof}
Use Proposition \ref{prop-residue} with $p(x)=(1+y+2\sqrt{y} x)^n$.
We need to compute the residue of
$$
\frac{p(\frac{z+1/z}{2})(z-1/z)^2}{zf(z,a,b,c,d)}
$$
at $z=0$ with $
f(z,a,b,c,d)=(1-az)(1-a/z)(1-bz)(1-b/z)(1-cz)(1-c/z)(1-dz)(1-d/z).$ 
Now we substitute 
$a\rightarrow a/\sqrt y$, $b\rightarrow b\sqrt y$, $c\rightarrow c/\sqrt y$, $d\rightarrow d\sqrt y$.
Since
$$
\frac{1}{f(z,A,B,C,D)}=
z^4/ABCD\sum_{s=0}^\infty H_s(A,B,C,D) z^s,
$$
we need the residue of 
$$\frac{(\sqrt{y}(z+1/z)+y+1)^n(z-1/z)^2
z^3}{ABCD} \sum_{s=0}^\infty H_s(A,B,C,D) z^s
$$
at $z=0$ or equivalently the coefficient of $z^n$ in
$$
\frac{(\sqrt{y}+z)^n(1+\sqrt{y}z)^n(z-1/z)^2
z^4}{ABCD} \sum_{s=0}^\infty H_s(A,B,C,D) z^s,
$$
which is
$$
\sum_{k=0}^n {n \choose k}\sqrt{y}^{n-k}
\sum_{j=0}^n {n \choose j}\sqrt{y}^j
(H_{n-2-k-j}-2H_{n-4-k-j}+H_{n-6-k-j}).
$$
\end{proof}

\begin{example}
$$
\begin{aligned}
Z_2(y;\alpha,\beta,\gamma,\delta;0)&=y^2\alpha^2 + y^2\alpha \delta + y^2\alpha^2 \delta + y^2\alpha \beta \delta 
+ y^2\alpha \delta^2 + y^2\alpha \delta \gamma + y\alpha \beta  + 
y\alpha^2\beta  \\&
+ y\alpha \beta^2  + y\beta \delta  + y\alpha \beta \delta  
+ y\alpha  \gamma  + y\alpha \beta  \gamma  +
  y\delta  \gamma  + y\alpha \delta  \gamma  + y\beta \delta  \gamma  \\
  &
  + y\delta^2  \gamma  + y\delta  \gamma^2  + 
 \beta^2  + \beta  \gamma + \alpha \beta  \gamma  + 
 \beta^2  \gamma + \beta \delta  \gamma + 
 \beta  \gamma^2.
\end{aligned}
$$
\end{example}

The following corollaries below follows directly from Theorem \ref{q0partfunc}.

\begin{corollary}\label{genfun0}
When $\alpha=\beta=1$ and $\gamma=\delta=q=0$, the 
generating function for the numbers 
$Z_n(y)=Z_n(y; \alpha,\beta,\gamma,\delta;q)$ is 
$$\sum_{n=0}^\infty Z_n(y) w^n=(1+t)(1+yt)
$$ where 
$w=\frac{t}{(1+t)(1+yt)}$.
Note that this is the generating function of the Narayana numbers \cite{Williams}.
\end{corollary}

\begin{corollary}\label{genfun}
When $\alpha=\beta=\gamma=1$ and $\delta=q=0$, the 
generating function for the numbers 
$Z_n(y)=Z_n(y; \alpha,\beta,\gamma,\delta;q)$ is 
$$\sum_{n=0}^\infty Z_n(y) w^n=
\frac{(1+t)(1+yt)}{(1-t-t^2)},$$ where 
$w=\frac{t}{(1+t)(1+yt)}$.
Note that when $y=1$, $Z_n(y)$  is the sequence
Sloane A026671, and when $y=0$ $Z_n(y)$  is the sequence odd Fibonacci numbers.
\end{corollary}

\begin{example}\label{genfunexample}
Set $\alpha=\beta=\gamma=1$ and $\delta=q=0$.  The polynomials
$Z_n(y)$ are
\begin{align*}
Z_0 &= 1 \\
Z_1 &= 2 + y\\
Z_2 &= 5 + 5y + y^2\\
Z_3 &= 13 + 20y + 9y^2 + y^3\\
Z_4 &= 34 + 72y + 52y^2 + 14y^3 + y^4\\
Z_5 &= 89 + 242y + 245y^2 + 110y^3 + 20y^4 + y^5
\end{align*}
\end{example}

\begin{corollary}\label{Fibonacci}
When $\alpha=\beta=\gamma=\delta=1$ and $q=0$, the 
generating function for the numbers 
$Z_n(y)=Z_n(y; \alpha,\beta,\gamma,\delta;q)$ is 
$$\sum_{n=1}^\infty Z_n(y) w^n=\frac{2(1+y)t(1+t)(1+yt)}{(1-t-t^2)(1-yt-y^2t^2)},$$ where 
$w=\frac{t}{(1+t)(1+yt)}$.
\end{corollary}

\begin{example}\label{genfunexample1}
Set $\alpha=\beta=\gamma=\delta=1$ and $q=0$.  The polynomials
$Z_n(y)$ are
$$
\begin{aligned}
Z_1&=2(1+y)\\
Z_2&=6(1+y)^2\\ 
Z_3&=2(1+y)(8+15y+8y^2)\\
Z_4&=2 (1+y)^2 (21+34 y+21 y^2)\\
Z_5&=2 (1+y) (55+181 y+253 y^2+181 y^3+55 y^4)\\
\end{aligned}
$$
\end{example}
\begin{remark}
Corollaries \ref{genfun0}, \ref{genfun} and \ref{Fibonacci} can also be proved
by induction. We can write recurrences for $Z_{n,k,j}$ the number of tableaux
of size $n$  with $k$ rows indexed by $\alpha$ or $\gamma$, and $j$ entries equal to $\alpha$
or $\delta$ on the diagonal. From this we  write a
functional equation for $Z(w,t,y)=\sum_{n,k,j} Z_{n,k,j} w^n t^k y^j$. Finally
$t$ is used as a catalytic variable and we extract $Z(w,1,y)$.
\end{remark}

\section{Combinatorics of staircase tableaux}

The motivation for defining staircase tableaux in \cite{CW3, CW4}
was to give a combinatorial
formula for the stationary distribution of the ASEP with all parameters
$\alpha,\beta,\gamma,\delta,q$ general.  Such a formula had already 
been given in \cite{CW1} using permutation tableaux, when $\gamma=\delta=0$.
Therefore it follows that
the set of staircase tableaux containing only $\alpha$'s or $\beta$'s 
are in bijection with both the
permutation tableaux coming from Postnikov's work \cite{Postnikov, SW},
and the alternative tableaux introduced by Viennot \cite{Viennot}.
These bijections are explained in 
\cite[Section 9]{CW4}.
As a consequence, the staircase tableaux of size $n$
with only $\alpha$'s and $\beta$'s
are in bijection with permutations on 
$n+1$ letters \cite{SW, CN, Burstein}.
Moreover, one can interpret the parameter $q$ as counting the number of
{\it crossings} or the number of patterns $31-2$ in the permutation \cite{Corteel,SW}.

In this section we explore more of the combinatorial properties of 
staircase tableaux, and in particular, explain the formulas in Table 
\ref{partition-table}.

\subsection{Enumeration of staircase tableaux when $q=1$}

As before, we set $Z_n = \sum_{\T} \wt(\T)$,
where the sum is over all staircase tableaux of size $n$.
When $q=y=1$, the weighted sum of staircase tableaux of size $n$ 
factors as a product of $n$ terms.

\begin{theorem}  \label{partition}
When $q=y=1$,
\begin{equation*}
Z_n(1;\alpha,\beta,\gamma,\delta;1) = \prod_{j=0}^{n-1} (\alpha+\beta+\gamma+\delta+j(\alpha+\gamma)(\beta+\delta)).
\end{equation*}
\end{theorem}

\begin{proof}
When $q=y=1$, it's clear from the definition of staircase tableaux 
that $Z_n(1;\alpha,\beta,\gamma,\delta; 1) = 
Z_n(1;\alpha+\gamma, \beta+\delta,0,0;1).$ 
The result then follows from the fact that 
\begin{equation}\label{part-ab}
Z_n(1; \alpha, \beta, 0,0;1)=
\prod_{j=0}^{n-1}(\alpha+\beta+j \alpha \beta),
\end{equation}
which was proved combinatorially 
(using the language of permutation tableaux)
in \cite{CN}.  We will give another proof of Equation
\ref{part-ab} in Section \ref{trees}.
\end{proof}

\begin{remark}
Note that 
Theorem \ref{partition}
and Theorem \ref{NewThm} immediately imply a result of 
Uchiyama, Sasamoto, and Wadati \cite{USW},
which is that the partition function of the ASEP 
with $\alpha, \beta, \gamma, \delta$ general and $q=1$
is given by
\begin{equation*}
\prod_{j=0}^{n-1} (\alpha+\beta+\gamma+\delta+j(\alpha+\gamma)(\beta+\delta)).
\end{equation*}
They prove this result by noting that when $q=1$,
the partition function $Z_n$ of the ASEP is equal to the $n$th moment
of the Laguerre polynomials $L_n^{(\lambda)}(x)$ with 
$$\lambda = \frac{\alpha+\beta+\gamma+\delta}{(\alpha+\gamma)(\beta+\delta)} - 1,$$
defined by 
$L_0^{(\lambda)}(x)=1$, $L_{-1}^{(\lambda)}(x)=0$, and 
$$(n+1) L_{n+1}^{(\lambda)}(x) - (2n+\lambda+1-x) L_n^{(\lambda)}(x)
+(n+\lambda)L_{n-1}^{(\lambda)}(x)=0.$$ 
\end{remark}

\subsection{Staircase tableaux and trees}\label{trees}

We begin by describing a bijective approach to understanding 
staircase tableaux that uses an underlying forest structure
of the tableaux.

Let $D(\ct)$ be the diagram of a staircase tableau of $\alpha$'s and
$\beta$'s. Regard the entries $\alpha$ and $\beta$ as vertices of a
graph. For each nondiagonal vertex $v$, regard the two nearest
vertices directly to the right and directly below $v$ as the children
of $v$, called the \emph{row child} and \emph{column child},
respectively. In this way $D(\ct)$ becomes a \emph{complete rooted
  binary forest}, i.e., a forest for which every component is a rooted
tree, and every non-endpoint vertex has exactly two children. The
endpoints are just the diagonal vertices. We call such a forest a
\emph{staircase forest}. If the forest is a tree, then we call it a
\emph{staircase tree}. Figure~\ref{fig:t4} shows the six staircase
trees of size 4. Note that the children of any internal vertex $v$ of
a staircase forest have uniquely determined labels $\alpha$ or $\beta$,
viz., the child to the right of $v$ is labelled $\alpha$, while the
child below $v$ is labelled $\beta$. We can label each root either
$\alpha$ or $\beta$ without changing the property of being a staircase
tableau.

\begin{figure}
\centering
 \centerline{\includegraphics[width=10cm]{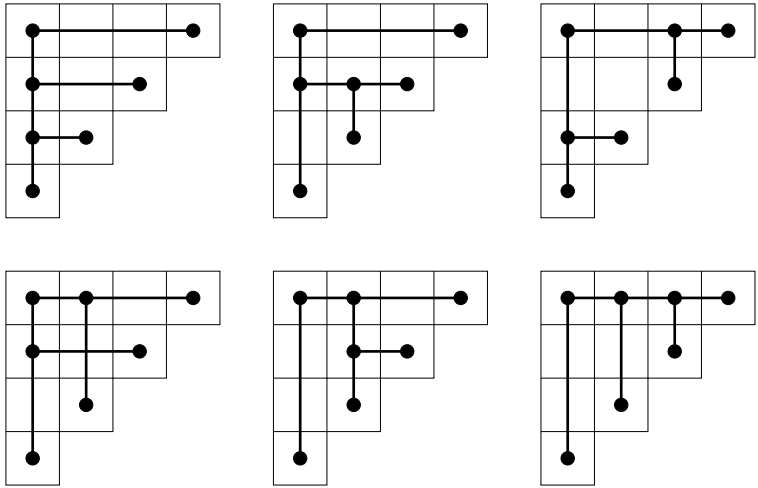}}
\caption{The staircase trees of size 4}
\label{fig:t4}
\end{figure}

The first step in enumerating $(\alpha,\beta)$-staircase tableaux by
this forest approach is the case where the forest is a tree. Let
$t(n)$ denote the number of staircase trees of size $n$.  The root $u$
must be in the upper-left corner (the $(1,1)$-entry).  Let
$v_1,\dots,v_n$ be the diagonal vertices, from top to bottom. The
row subtree of $u$ (i.e., the subtree whose root is the row child of
$u$) can have any nonempty subset $S$ of the diagonal
vertices as endpoints, except for the conditions $v_1\in S$ and
$v_n\not\in S$. If the row subtree has $i$ endpoints, then there are
$\binom{n-2}{i-1}$ ways to choose them, and then $t(i)$ ways to choose
the subtree itself.  Similarly there are $t(n-i)$ choices for the
column subtree. Hence
   \beq t(n) =\sum_{i=1}^{n-1}\binom{n-2}{i-1}t(i)t(n-i),
     \label{eq:tnrec} \eeq
with the initial condition $t(1)=1$. The solution to this recurrence is
clearly $t(n)=(n-1)!$, since the above sum will then have $n-1$ terms,
all equal to $(n-2)!$.

The formula $t(n)=(n-1)!$ shows that $t(n)$ is equal to the number of
$n$-cycles in the symmetric group $\sn$, while
equation~\eqref{eq:tnrec} shows that the number of staircase trees of
size $n$ whose row subtree has $i$ endpoints, $1\leq i\leq n-1$, is
$(n-2)!$, independent of $i$. From this observation it is
straightforward to give a bijection between staircase trees and
$n$-cycles. 

Alternatively, one can give an explicit bijection from staircase
trees to cycles as follows.
If $v_i$ is incident to a vertical edge, then
travel north from $v_i$ as far as possible along edges of the 
staircase tree, then
take a ``zig-zag''
path east-south,  traveling on edges of the tree
east and south and 
turning at each new vertex.  This path will
terminate at some vertex $v_j$ where $j<i$;
set $\pi(i) = j$.  Similarly, if $v_i$ is incident to a horizontal
edge, then 
travel west from $i$ as far as 
possible along edges of the tree, then 
take a zig-zag path south-east, traveling on edges of the tree south and east,
and turning at each new vertex.
This path will terminate at some $v_j$ for $j>i$; 
set $\pi(i) = j$.
The result will be a permutation $\pi$ which is a single cycle.
Using this bijection, the cycles associated to the first row
of Figure \ref{fig:t4} are, from left to right,
$(1234)$, $(1324)$, and $(1342)$, and the cycles associated to the 
second row are $(1243)$, $(1423)$, and $(1432)$.

Let us now consider staircase forests $F$. We obtain such a forest by
choosing a partition $\{B_1,\dots,B_k\}$ of the endpoints and then
for each block $B_i$ choosing a staircase tree whose endpoints are
$B_i$. Since a staircase tree with endpoints $B_i$ is equivalent to a
cycle on the elements of $B_i$, we are just choosing a permutation of
the endpoints. Hence there are $n!$ staircase forests of size $n$.

With little extra difficulty we can handle the labels $\alpha,
\beta$. Half the non-root vertices are row children, while half are
column children. If $F$ is a staircase forest and $T$ a component of
$F$ with $k$ endpoints, then $T$ has $k-1$ row children, all labelled
$\alpha$, and $k-1$ column children, all labelled $\beta$. The root is
labelled either $\alpha$ or $\beta$. Identifying the components of a
staircase forest with the cycles of a permutation shows that
  $$ \sum_{\ct}\wt(\ct) =\sum_{w\in\sn}\wt(w), $$
where the first sum ranges over all $(\alpha,\beta)$-staircase
tableaux of size $n$, while in the second sum we define
   $$ \wt(w) =\prod_C (\alpha+\beta)(\alpha\beta)^{\#C-1}, $$
where $C$ ranges over all cycles of $w$. For instance, if
$w=(1,3,6)(2,8)(4,9,7)(5)$ (disjoint cycle notation), then
  $$ \wt(w) = (\alpha+\beta)^4(\alpha\beta)^5. $$

A basic enumerative result on cycles of permutations states that if
$\kappa(w)$ denotes the number of cycles of $w$ then
  \beq F_n(x):= \sum_{w\in\sn} x^{\kappa(w)}=x(x+1)\cdots(x+n-1).
    \label{eq:fnx} \eeq
Hence
 \bea \sum_{\ct}\wt(\ct)  & = & \sum_{w\in \sn}
     (\alpha+\beta)^{\kappa(w)}(\alpha\beta)^{n-\kappa(w)}
     \nonumber\\
    & = & (\alpha\beta)^n F_n((\alpha+\beta)/\alpha\beta)
     \nonumber\\ & = &
    (\alpha+\beta)(\alpha+\beta+\alpha\beta)(\alpha+\beta+
    2\alpha\beta)\cdots(\alpha+\beta+(n-1)\alpha\beta).
    \label{eq:ctwt} \eea
Equation~\eqref{eq:fnx} has (at least) two bijective proofs
\cite[Prop.~1.3.4]{ec1}, so these two proofs carry over to bijective
proofs of equation~\eqref{eq:ctwt}. 

Note that Theorem \ref{partition} implies that there are
$4^n n!$ staircase tableaux of size $n$.  One can 
give a simple bijection ${\Phi}$ from staircase tableaux 
to {\it doubly signed permutations} -- that is, permutations
where each position is decorated by two signs.
The underlying permutation associated to a staircase tableaux
is just the permutation associated to its staircase forest
(after replacing $\gamma$'s and $\delta$'s with $\alpha$'s and 
$\beta$'s, respectively).  The first sign associated to 
position $i$ is $+$ if the $i$th diagonal box contains an 
$\alpha$ or $\delta$, and is $-$, otherwise.  
The second sign which we associate to position $i$
depends on the $i$th diagonal vertex
and either the topmost vertex in the $i$th column
or the leftmost vertex in the $i$th row.
More specifically, if the $i$th diagonal vertex is $\alpha$ or $\gamma$ and the leftmost
vertex of the $i$th row is $\alpha$ or $\delta$, we assign a $+$.
If the $i$th diagonal vertex is $\alpha$ or $\gamma$ and the leftmost
vertex of the $i$th row is $\beta$ or $\gamma$, we assign a $-$.
On the other hand, if the $i$th diagonal vertex is $\beta$ or $\delta$
and the topmost vertex of the $i$th column is $\alpha$ or $\beta$
we assign a $+$.
If the $i$th diagonal vertex is $\beta$ or $\delta$
and the topmost vertex of the $i$th column is $\gamma$ or $\delta$
we assign a $-$.

\begin{remark}
Philippe Nadeau exhibits a simple recursive structure for alternative tableaux
(staircase tableaux with no $\gamma$ and $\delta$) \cite{Nadeau}. 
The results of this section can be derived with Nadeau's techniques.
For example, Proposition 3.5 of \cite{Nadeau} implies that
there exist  $4^n n!$ staircase tableaux of size $n$.
James Merryfield also gave a bijective proof that the 
staircase tableaux have cardinality $4^n n!$ \cite{Merryfield}.
\end{remark}

As a slight variant of equation~\eqref{eq:fnx}, consider the problem
of counting the number $g(n)$ of $(\alpha,\beta,\gamma)$-staircase
tableaux of size $n$. Substituting $\alpha+\gamma$ for $\alpha$ and
setting $\alpha=\beta=\gamma=1$ (or just setting $\alpha=2$ and
$\beta=1$ in \eqref{eq:fnx}) gives
  $$ g(n) = 3\cdot 5\cdots (2n+1) = (2n+1)!!, $$
the number of complete matchings on a $(2n+2)$-element set. By the
interpretation in terms of cycles, we are counting permutations in
$\sn$ where the least element in each cycle in 3-colored and the
remaining elements are 2-colored. The third proof of
\cite[Prop.~1.3.4]{ec1} gives a bijective proof of this result by
first making three choices, then five choices, up to $(2n+1)$
choices. It is easy to encode these choices by a complete matching on
$[2n+2]$, thereby giving a bijection between
$(\alpha,\beta,\gamma)$-staircase tableaux and matchings.

\subsection{Enumeration of staircase tableaux of a given type}

Recall from equation (\ref{type}) that  $Z_{\sigma}(\alpha,\beta,\gamma,\delta;q)$
is the generating polynomial for the staircase tableaux of type $\sigma$; 
here $\sigma$ is a word
in $\{\bullet,\circ\}^n$.  By Theorem \ref{NewThm}, the steady 
state probability that the ASEP is at state $\sigma$ is proportional
to $Z_{\sigma}(\alpha,\beta,\gamma,\delta;q)$.  Therefore it is desirable to have explicit formulas
for $Z_{\sigma}(\alpha,\beta,\gamma,\delta;q)$.

We do not have an explicit formula for
$Z_{\sigma}(\alpha,\beta,\gamma, \delta;q)$ which works for arbitrary values
of the parameters.
However, a few special cases are known.  In particular, we will discuss
the following:
\begin{itemize}
\item an explicit formula for $Z_{\sigma}(1,1,1,1;1)$;
\item an explicit formula for $Z_{\sigma}(1,1,0,0;q)$;
\item relations satisfied by $Z_{\sigma}(\alpha,\beta,\gamma,\delta;q)$;
\item a recurrence for $Z_{\sigma}(\alpha,\beta,\gamma,0;q)$.
\end{itemize}

\subsubsection{An explicit formula for $Z_{\sigma}(1,1,1,1;1)$}
\begin{proposition}
For any word $\sigma$ in $\{\bullet,\circ\}^n$, 
$Z_{\sigma}(1,1,1,1; 1) = 2^n n!$.  In other words, there are $2^n n!$ staircase tableaux
of each type.
\end{proposition}

\begin{proof}
This follows directly from the definition of staircase tableaux.
Fix an arbitrary word $\sigma$ in $\{\bullet,\circ\}^n$.
Let us show that the staircase tableaux of type $\sigma$ are 
in bijection with the staircase tableaux of type $\bullet^n$.
Given a tableau $T$ of type $\sigma$, we map it to a tableau
$g(T)$ of type $\bullet^n$, by replacing every diagonal box
which contains a $\beta$ with a $\delta$, and by replacing
every diagonal box which contains a $\gamma$ with an $\alpha$.
This is clearly a bijection.  Since there are $2^n$ words
of length $n$ in $\{\bullet,\circ\}^n$, and the total number of staircase
tableaux is $4^n n!$, there must be $2^n n!$ staircase tableaux of each
type in $\{\bullet,\circ\}^n$.
\end{proof}

\subsubsection{An explicit formula for $Z_{\sigma}(1,1,0,0;q)$}

There is also an explicit formula when $\gamma=\delta=0$, and 
$\alpha=\beta=1$, which was found in \cite{NTW} (though stated
in terms of permutation tableaux).
We first need to give a few definitions.

A {\it composition} of $n$ is  a list of positive integers which 
sum to $n$. If $I=(i_1,\dots,i_r)$ is a composition, let
$\ell(I)=r$ be its number of parts.  
The \emph{descent set} of $I$ is
$\Des(I) = \{ i_1,\ i_1+i_2, \ldots , i_1+\dots+i_{r-1}\}$.
We say that a composition $J$ is {\it weakly coarser} than $I$, 
denoted $J \preceq I$, 
if $J$ is obtained from $I$ by merging some parts of $I$.  For example,
the compositions which are (weakly) coarser than the composition $(3,4,1)$
are $(3,4,1)$, $(7,1)$, $(3,5)$, and $(8)$.

Given a word $\sigma$ in $\{\bullet,\circ\}^n$, we associate to it
a composition $I(\sigma)$ as follows.
Read $\sigma$ from right to left, and list the lengths of the 
consecutive blocks of $\circ$'s, between the right end of $\sigma$ 
and the rightmost $\bullet$, between two $\bullet$'s, and 
between the leftmost $\bullet$ and the left end of $\sigma$. 
This gives $I'(\sigma)$.  For example,
if $\sigma = \bullet \circ \circ \circ \bullet \circ \circ$ then 
$I'(\sigma) = (2,3,0)$.  Then 
we define $I(\sigma)$ by adding $1$ to each entry of $I'(\sigma)$.
So in this case, $I(\sigma) = (3,4,1)$.

We also define a relative of the $q$-factorial function: 
we define $\QStat$ as a function of any composition by 
\begin{equation*}
\QStat(j_1,\dots,j_p) := [p]_q^{j_1} [p-1]_q^{j_2} \dots [2]_q^{j_{p-1}}
[1]_q^{j_p}.
\end{equation*}
Here $[p]_q := 1+q+\dots+q^{p-1}$.

Finally, if $I\succeq J$, we define the statistic
$st(I,J)$ by
\begin{equation*}
st(I,J):=
\#\{(i,j)\in \Des(I)\times\Des(J) | i\leq j\}.
\end{equation*}

\begin{theorem}\cite[Theorem 4.2]{NTW}
Let $\sigma$ be any word in $\{\circ, \bullet\}^n$, and let 
$I:=I(\sigma)$ be the composition associated to $\sigma$. Then
\begin{equation*}
\label{eqptA}
Z_{\sigma}(1,1,0,0;q)
= \sum_{J \preceq I} (-1/q)^{l(I)-l(J)} q^{-st(I,J)} \QStat(J).
\end{equation*}
\end{theorem}

\begin{example}
When $\sigma = \bullet \circ \circ \circ \bullet \circ \circ$,
we have $I:=I(\sigma) = (3,4,1)$.
The compositions coarser than $I$ are
$(3,4,1)$, $(7,1)$, $(3,5)$, and $(8)$, so
we get
\begin{align*}
Z_{\sigma}(1,1,0,0;q) 
&= q^0 q^{-3} [3]_q^3 [2]_q^4 [1]_q^1
                                - q^{-1} q^{-2} [2]_q^7 [1]_q^1
                                - q^{-1} q^{-1} [2]_q^3 [1]_q^5
                                + q^{-2} q^0 [1]_q^8\\
&=
q^7  + 7 q^6  + 24 q^5  + 52 q^4  + 76 q^3  + 75 q^2  + 47 q + 15.
\end{align*}
This is the generating polynomial for staircase tableaux of 
type $\bullet \circ \circ \circ \bullet \circ \circ$ when
$\alpha=\beta=1$ and $\gamma=\delta=0$.
\end{example}

\subsubsection{Relations satisfied by $Z_{\sigma}(\alpha,\beta,\gamma,\delta;q)$}\label{relations}
In this subsection we will use $Z_{\sigma}$ as shorthand for  
$Z_{\sigma}(\alpha,\beta,\gamma,\delta;q)$.
We will recall here some relations satisfied by $Z_{\sigma}$ that were proved
in \cite{CW3, CW4}.
Given any word $\sigma$ in the alphabet
$\{\bullet,\circ\}$, let $\ell(\sigma)$ denote the length of $\sigma$.  

\begin{theorem}\cite{CW3, CW4}\label{CWTh}
Let $\alpha,\beta,\gamma,\delta, q$ be arbitrary parameters, and 
let $\lambda_n$ be defined by 
$\lambda_n = \alpha \beta -\gamma \delta q^{n-1}$ for $n\geq 1$.  
Let $\sigma_1, \sigma_2, \sigma$ be arbitrary words in the alphabet
$\{\bullet,\circ\}$.  Then we have the following relations among
the $Z_{\sigma}=Z_{\sigma}(\alpha,\beta,\gamma,\delta;q)$.
\begin{align}
Z_{\sigma_1 \bullet \circ \sigma_2} - qZ_{\sigma_1 \circ \bullet \sigma_2}
&= \lambda_{\ell(\sigma_1)+\ell(\sigma_2)+2} (Z_{\sigma_1 \bullet \sigma_2} +
Z_{\sigma_1 \circ \sigma_2}).\\
\alpha Z_{\circ \sigma} - \gamma Z_{\bullet \sigma} &=\lambda_{\ell(\sigma)+1} Z_{\sigma}.\\
\beta Z_{\sigma \bullet} - \delta Z_{\sigma \circ} &= \lambda_{\ell(\sigma)+1} Z_{\sigma}.
\end{align}
\end{theorem}
Note that the proof of Theorem \ref{CWTh} in \cite{CW3,CW4} used
a complicated induction and was not very combinatorial.

\subsubsection{A recurrence for $Z_{\sigma}(\alpha,\beta,\gamma,0;q)$.}\label{rec}

However, when $\delta=0$, Theorem \ref{CWTh} simplifies, and we can 
give a purely combinatorial proof using staircase tableaux.
Throughout this section our staircase tableaux will be assumed to have no 
$\delta$'s, and 
we will abbreviate 
$Z_{\sigma}(\alpha,\beta,\gamma,0;q)$ by $Z_{\sigma}$.

\begin{theorem}\label{explicit-d}
Let $\sigma, \sigma_1, \sigma_2$ be arbitrary words in the alphabet
$\{\bullet,\circ\}$.  Then we have the following.
\begin{align} \label{first}
Z_{\sigma_1 \bullet \circ \sigma_2} &= 
qZ_{\sigma_1 \circ \bullet \sigma_2} +
\alpha \beta (Z_{\sigma_1 \bullet \sigma_2}+Z_{\sigma_1 \circ \sigma_2}),\\
\label{second} \alpha Z_{\circ \sigma} &= \gamma Z_{\bullet \sigma} + \alpha \beta Z_{\sigma},\\
\label{third} Z_{\sigma \bullet} &= \alpha  Z_{\sigma}.
\end{align}
\end{theorem}

\begin{proof}
This recurrence is best explained with pictures.
We begin with equations (\ref{third}) and (\ref{second}), which are
easiest to prove.
To prove (\ref{third}), it suffices to note that any staircase
tableau (with no $\delta$'s) whose type ends with $\bullet$ must
have an $\alpha$ in its lower left square.  The weight of such a 
tableau is $\alpha$ times the weight of the tableau obtained from it
by deleting the leftmost column.  See the left part of Figure 
\ref{recurrence1}.  Equation (\ref{third}) follows.

\begin{figure}[h]
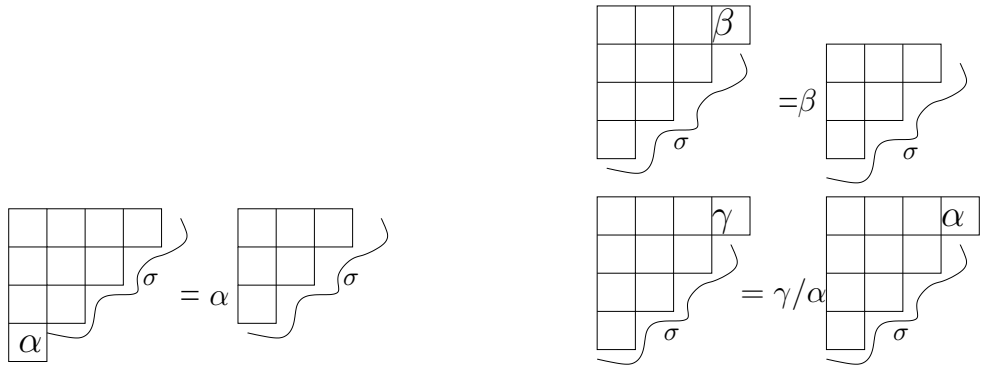

\input{Recurrence1.pstex_t} \hspace{6em} \input{Recurrence2.pstex_t}
\caption{The left and right parts of the picture prove equations
(\ref{third}) and (\ref{second})}
\label{recurrence1}
\end{figure}

To prove (\ref{second}) note that any staircase tableau
whose type begins with $\circ$ must have a $\beta$ or a $\gamma$
in its upper right square.  If it has a $\beta$ there, then the weight
of that tableau is equal to $\beta$ times the weight of the tableau
obtained by deleting the topmost row.  Alternatively, if it has
a $\gamma$ there, then its weight is equal to $\frac{\gamma}{\alpha}$
times the weight of the tableau obtained from it by replacing the 
$\gamma$ with an $\alpha$.  See the left right of Figure 
\ref{recurrence1}.

\begin{figure}[h]
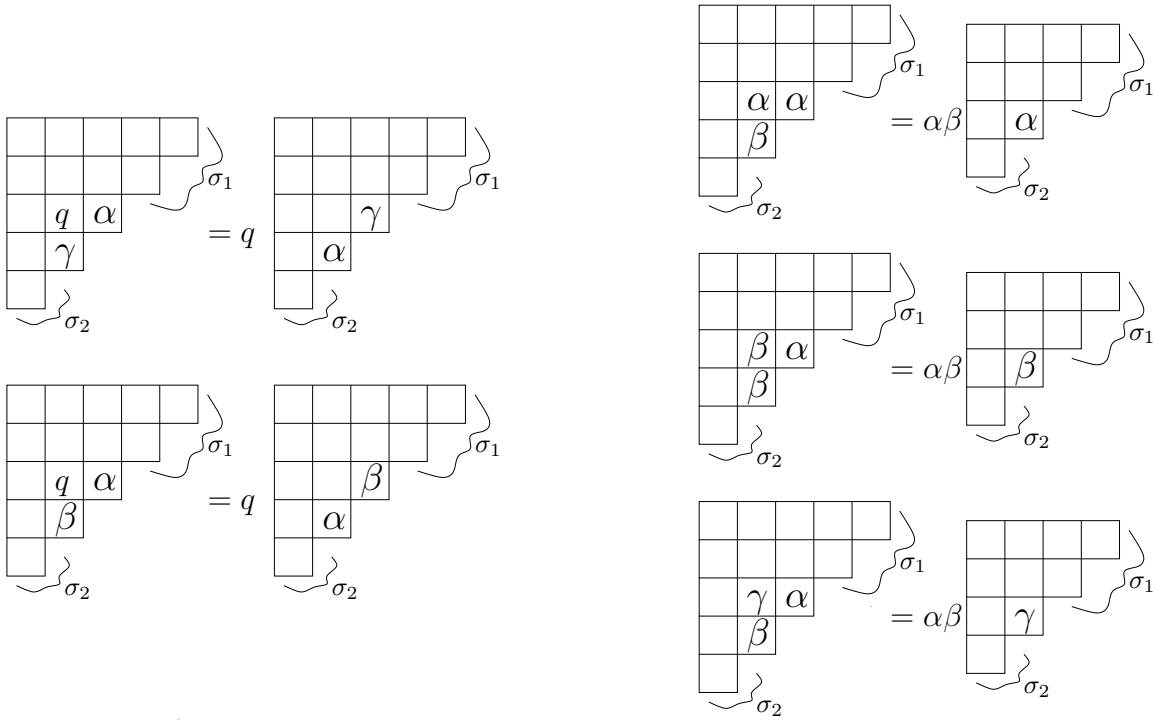

\input{Recurrence3.pstex_t} \hspace{6em} \input{Recurrence4.pstex_t}
\caption{This picture proves equation (\ref{first}).}
\label{recurrence2}
\end{figure}

To prove (\ref{first}), note that if a staircase tableau has type
$\sigma_1 \bullet \circ \sigma_2$, then its two diagonal boxes
corresponding to the $\bullet \circ$ must be either
$\alpha \gamma$ or $\alpha \beta$.  If the two boxes are
$\alpha \gamma$, then the box above the $\gamma$ and left of the 
$\alpha$ will get filled with $q$.  If the two boxes
are $\alpha \beta$, then the box above the $\beta$ and left of the
$\alpha$ may be filled with either a $q, \alpha, \beta$ or $\gamma$.  These five
possibilities are shown in Figure \ref{recurrence2}.

As shown in the left of Figure \ref{recurrence2}, if that third box
is a $q$, then the generating polynomial for 
such staircase tableaux of type $\sigma_1 \bullet \circ \sigma_2$
is equal to $q$ times the generating polynomial for staircase tableau
of type $\sigma_1 \circ \bullet \sigma_2$.  One can prove this bijectively
by taking the two columns above the $q$ and $\alpha$ and swapping them;
and by taking the two rows left of the $q$ and $\gamma$ (respectively
$q$ and $\beta$) and swapping them.  (The box filled with the $q$ will become
a box filled with $u=1$.)
 
On the other hand, if the two diagonal boxes are $\alpha$ and $\beta$,
and the box above the $\beta$ and left of the $\alpha$ is filled with 
either $\alpha, \beta,$ or $\gamma$, then the weight of this tableau
is equal to $\alpha \beta$ times the weight of the tableau obtained
by deleting the column with the $\alpha$ in the diagonal box,
and deleting the row with the $\beta$ in the diagonal box.
This completes the proof of (\ref{first}).
\end{proof}

\subsection{More factorizations of the partition function}

Theorem \ref{CWTh} is very useful for proving various factorizations
of the partition function.

\begin{proposition}\label{d-b}
$$Z_n(1;\alpha,\beta,\gamma,-\beta;q) = \prod_{j=0}^{n-1} (\alpha+ q^j \gamma).$$
\end{proposition}

\begin{proof}
We use Theorem \ref{CWTh}, with $\delta=-\beta$.  In this case
we have 
$\lambda_n = \beta (\alpha+q^n \gamma)$, and so 
$Z_{\sigma \bullet} + Z_{\sigma \circ} = (\alpha+q^{\ell(\sigma)+1} \gamma) Z_{\sigma}.$
We now use induction on $n$.  Note that 
$Z_n = \sum_{\sigma} Z_{\sigma}$, where the sum is over all words
$\sigma \in \{\bullet,\circ \}^n$.  
Then $Z_{n+1} = \sum_{\sigma} (Z_{\sigma \bullet} + Z_{\sigma \circ}) = 
\sum_{\sigma} (\alpha+q^n \gamma) Z_{\sigma} = (\alpha+q^n \gamma) Z_n$.  
The results now follows by induction.
\end{proof}

\begin{proposition}\label{y-1}
$$Z_n(-1;\alpha,\beta,\gamma,\beta;q)=(-1)^n \prod_{j=0}^{n-1} (\alpha- q^j \gamma).$$
\end{proposition}
\begin{proof}
Exercise.  The proof is analogous to the proof
of Proposition \ref{d-b}.
\end{proof}

\begin{proposition}\label{0}
$$Z_n(y; \alpha,\alpha,\alpha,\alpha;-1) =0 \text{ for }n\geq 3.$$ 
\end{proposition}

\begin{proof}
We use Theorem \ref{explicit-Z}.
Our choice of specialization makes the sum over $j$ equal to $0$
unless $k=0$ or $k=1$.  (To see this, 
split the sum over $j$ into sums over even and odd $j$, and
consider even and odd $k$.)
The $k$-sum has the term
$(q^k;q)_{n-k}$, which is zero for $q=-1$ when
either $k=0$ and $n>0$, or 
$k=1$ and $n>2$.
So if $n>2$ we get $0$.
\end{proof}

\begin{remark} 
Note that one can generalize the preceeding proposition
and prove that $Z_n(y; \alpha,\alpha,\alpha,\alpha;q) =0 $ for  $n> 2m$ 
and $q^m=-1$.
\end{remark}

\subsection{Enumeration of staircase tableaux when $\beta=1$ and $\delta=0$}

As we've seen in Section \ref{rec}, the combinatorics of staircase
tableaux becomes a bit simpler when $\delta=0$.  In this section
we explore the combinatorics when in addition we impose $q=1$.
By Theorem \ref{partition}, the generating polynomial for 
staircase tableaux with no $\delta$'s is 
$$
\prod_{j=0}^{n-1} (\alpha+\beta+\gamma+j\beta(\alpha+\gamma)).
$$
Therefore the number of staircase tableaux of size with no $\delta$ is $(2n+1)!!=(2n+1)\cdot(2n-1)\cdot\ldots \cdot 3\cdot 1$.  Since 
$(2n+1)!!$ is the number of perfect matchings of the set 
$\{1,2,\dots,2n+2\}$, this implies the following.

\begin{corollary}
There exists a bijection between the staircase tableaux
of size $n$ with no $\delta$
and the perfect matchings of $\{1,2,\ldots ,2n+2\}$.
\end{corollary}
We gave the sketch of a bijective proof of this result in Section \ref{trees}.
Now let us study the combinatorics of those tableaux with no $\delta$
in the case $\beta=1$.  Let
$Z_n(\alpha,\gamma;q)=Z_n(1;\alpha,1,\gamma,0;q)$.

\begin{proposition}
The generating polynomial 
$Z_n(\alpha,\gamma;q)$ of staircase tableaux of size $n$ 
is equal to the generating polynomial
of weighted Dyck paths of length $2n+2$ where the North-East
steps get weight 1 and the South East steps starting at height $i$
have weight 
\begin{itemize}
\item $(\alpha+\gamma q^i)[i+1]_q$ if $k=2i+2$
\item $q^i+(\alpha+\gamma q^i)[i]_q$ if $k=2i+1$.
\end{itemize}
\label{gp}
\end{proposition}

\begin{proof}
From Corollary \ref{Dennis}, we know that $Z_n$ with $\delta=0$
is a factor times the moments of the
orthogonal polynomials with 
$$
b_n=\frac{2+q^n((a+b+c)+(1-q^{n-1}(1+q))abc)}{1-q}$$ and
$$\lambda_n=\frac{(1-q^n)(1-q^{n-1}ab)(1-q^{n-1}ac)(1-q^{n-1}bc)}{(1-q)^2}.$$
If $\beta=1$, then $b=-q$. We get that
$Z_n$ is exactly equal to the moments of the polynomials with
$$b_n= [n+1]_q(\alpha+\gamma q^{n})+(\alpha+\gamma q^n)
[n]_q+q^n,\ \ {\rm and}\
\lambda_n=[n]_q(\alpha+\gamma q^{n-1})((\alpha+\gamma q^n)
[n]_q+q^n).$$ 
Now we use a result on page 46 of \cite{Chihara}
which says that  if $G_n(x)$ are
orthogonal polynomials with $b_n$ and $\lambda_n$ arbitrary,
then $H_{2n+1}(x)=xG_n(x^2)$
are orthogonal polynomials with 
$B_n=0$ and 
$\Lambda_{1}=1$, $\Lambda_{2n+2}=b_n-\Lambda_{2n+1}$
and $\Lambda_{2n+1}=\lambda_n/\Lambda_{2n}$. We get that
he generating polynomial of staircase tableaux of size $n$ $Z_n(\alpha,\gamma;q)$
is equal to the moments $\mu_{2n+2}$ of the orthogonals polynomials defined
by
$$
xH_n(x)=H_{n+1}(x)+\Lambda_n H_{n-1}(x)
$$ with
$\Lambda_{2n}=(\alpha+\gamma q^{n-1})[n]_q$ and
$\Lambda_{2n+1}=(\alpha+\gamma q^{n})[n]_q+q^n$. As Dyck paths are Motzkin paths
with no east steps, the proposition follows using Theorem \ref{moment-motzkin}.
\end{proof}

\noindent {\bf Remark.} When $q=\alpha=\gamma=1$, it is well known that these paths are in bijection
with perfect matchings of $\{1,2,\ldots ,2n+2\}$ \cite{FFV}. If the South East 
steps starting at height $i$
had weight $[i]_q$, these paths would correspond to moments
of the classical $q$-Hermite polynomials or perfect matchings counted
by crossings (see \cite{P} and references wherein). \\

We can now give a combinatorial interpretation of the preceeding proposition.
A matching of $\{1,\ldots ,2n+2\}$ is a sequence of $n+1$ mutually disjoint
edges $(i,j)$ with $1\le i<j\le 2n+2$.
Given an edge $e=(i,j)$, let $cross(e)$ be the number of edges $(\ell,k)$
such that $i<\ell<j<k$ and $nest(e)$  be the number of edges $(\ell,k)$
such that $\ell<i<j<k$. We define the f-crossing of an edge $e$
to be equal to $cross(e)$ if $nest(e)>0$ and $\lfloor cross(e)/2\rfloor$ otherwise.
We said that an edge is nested (resp. crossed) if $cross(e)<nest(e)$
(resp. if $cross(e)>nest(e)$).

\begin{theorem}
There exists a bijection between staircase tableaux of size $n$ with 
$j$ entries equal to $q$, $k$ entries equal to $\alpha$ and $\ell$ entries
equal to $\gamma$ and 
matchings of $\{1,\ldots ,2n+2\}$ where $j$ is the number of
f-crossings, $k$ is the number of nested edges and $\ell$ the number of crossed edges.
\end{theorem}
\begin{proof}
The proof is direct using the classical bijection between
labelled Dyck paths and matchings. See \cite{P} for example.
\end{proof}

\section{Open problems}

We conclude this paper with a list of open problems.

\begin{problem}
Give combinatorial proofs of Propositions \ref{d-b}, \ref{y-1},  and \ref{0},
using
appropriate involutions on staircase tableaux.
\end{problem}

\begin{problem}\label{explicit1}
Recall from equation (\ref{type}) that $Z_{\sigma}(\alpha,\beta,\gamma,\delta;q)$
is the generating polynomial for the staircase tableaux of type $\sigma$; 
here $\sigma$ is a word
in $\{\bullet,\circ\}^n$.  By Theorem \ref{NewThm}, the steady 
state probability that the ASEP is at state $\sigma$ is proportional
to $Z_{\sigma}(\alpha,\beta,\gamma,\delta;q)$.  Find an explicit formula
for $Z_{\sigma}(\alpha,\beta,\gamma,\delta;q)$.
\end{problem}

Problem \ref{explicit1} is probably quite difficult.
Problem \ref{explicit2} should be more tractable, however, since
one has a simple recurrence for $Z_{\sigma}(\alpha,\beta,\gamma,0;q)$ given
by Theorem \ref{explicit-d}.

\begin{problem}\label{explicit2}
Find an explicit formula for $Z_{\sigma}(\alpha,\beta,\gamma,0;q)$.
\end{problem}

\begin{problem}
Recall the definition of the bijection ${\Phi}$ from 
Secton \ref{trees}.
Find a statistic $s(\pi)$ on doubly signed permutations, which 
corresponds to the $q$ statistic on staircase tableaux via 
${\Phi}$.  More specifically, we require that 
$q^r$ is the maximal power of $q$ dividing $\wt(\T)$ if and only if 
$s({\Phi}(\T))=r$.  
\end{problem}

\begin{problem}
If we restrict ${\Phi}$ to the set of 
staircase tableaux of size $n$ of a given type, 
then for all $i$, the first sign associated to position $i$
is the same for all tableaux.  Therefore if we forget the first
sign, we get a bijection from the $2^n n!$ staircase tableaux
of a given type to signed permutations.  For any fixed type $\sigma$,
can one find a statistic $s_{\sigma}(\pi)$ 
on signed permutations which corresponds
to the $q$ statistic on the staircase tableaux of type $\sigma$?
\end{problem}

\begin{problem}
Find an explicit formula for 
$Z_n(y; \alpha, \beta, \gamma, \delta; q)$ from which it 
is obvious that 
$Z_n(y;\alpha,\beta,\gamma,\delta;q)$ 
is a polynomial with positive coefficients.  Such a formula
can be found when $\gamma=\delta=0$ \cite{JV}.
\end{problem}

\begin{problem}
Prove that
$Z_n(y;\alpha,\beta,\gamma,\delta;q)$ is equal to $y^nZ_n(1/y;\beta,\alpha,\delta,\gamma;q)$ by
exhibiting an involution on staircase tableaux.
\end{problem}


\begin{problem}\label{Sloane}
Give a simple bijection proving that 
the numbers $Z_n(y;\alpha,\beta,\gamma,\delta;q)$ 
are given by Sloane's sequence A026671 (enumerating certain lattice paths) when
$\alpha=\beta=\gamma=y=1$ and $\delta=q=0$.
Note that Corollary \ref{genfun} gives  
a (non-bijective) proof of this equality, by showing 
that the generating functions of both sets of numbers are equal.
See also Example \ref{genfunexample}.
\end{problem}

\begin{problem}
Can one find other formulas for the moments of Askey Wilson polynomials? In particular, is there
a formula that makes manifest the symmetry in $a,b,c,d$? 

This is possible when at least one of $a,b,c,d$ is 0.
In particular,  Josuat-Verg\`es \cite{JV} gave a strictly polynomial version of the
Askey Wilson moments when $c=d=0$.
\begin{eqnarray*}
2^n\mu_n(a,b)&=&
\sum_{t=0}^n\sum_{p=0}^{\lfloor (n-t)/2\rfloor}
\sum_{i=0}^p \left({n\choose p-i}-{n\choose p-i-1}\right)
(-1)^i q^{i+1\choose 2}\\
&&\times \left[\begin{array}{c} n-2p\\t\end{array}\right]_q
 \left[\begin{array}{c} n-2p+i\\i\end{array}\right]_qb^t a^{n-t-2p}.
 \end{eqnarray*}
This can be proved from our Theorem \ref{moments3}  with
$c=d=0$, using the $q$-binomial theorem, the binomial theorem
and the terminating $_2\phi_1(x)$ at $x=q$ \cite{GR}.

Using the same techniques, one can obtain a formula for $d=0$,
\begin{eqnarray*}
2^n\mu_n(a,b,c)&=&\sum_{t_1=0}^n \sum_{t_2=0}^n \sum_{p=0}^{\min[\lfloor (n-t_1)/2\rfloor,\lfloor (n-t_2)/2\rfloor]}
\sum_{F=\max[p,\lfloor (n-t_1-t_2)/2\rfloor]}^{min[n-p-t_1,n-p-t_2]}
\left({n\choose p}-{n\choose p-1}\right)\\
&&\times (-1)^{n-t_1-t_2-F-p}q^{n-t_1-t_2-F-p+1\choose 2}\\
&&\times \left[\begin{array}{c} t_1+F-p\\t_1\end{array}\right]_q
\left[\begin{array}{c} t_2+F-p\\n-F-p-t_1\end{array}\right]_q
\left[\begin{array}{c} n-F-p\\t_2\end{array}\right]_q b^{t_1}c^{t_2}a^{-n+t_1+t_2+2F}.
\end{eqnarray*}
This is obviously symmetric in $a$, $b$ and $c=0$.
One can get back the previous equation by setting $c=0$, so $t_2=0$, and then put
$F=n-t_1-p$.
\end{problem}

\begin{problem}
Find a combinatorial proof of the formula for $\mu_n(a,b,c)$ above.
\end{problem}

\end{document}